\documentclass[12pt]{amsart}
\usepackage{amssymb,mathrsfs}

\begin{document}


\title{Derivatives of $L$-functions}

\author{Jae-Hyun Yang}

\address{Yang Institute for Advanced Study
\newline\indent
Hyundai 41 Tower, No. 1905
\newline\indent
293 Mokdongdong-ro, Yangcheon-gu
\newline\indent
Seoul 07997, Korea}
\email{jhyang@inha.ac.kr\ \ or\ \ jhyang8357@gmail.com}

\newtheorem{theorem}{Theorem}[section]
\newtheorem{lemma}{Lemma}[section]
\newtheorem{proposition}{Proposition}[section]
\newtheorem{remark}{Remark}[section]
\newtheorem{definition}{Definition}[section]

\renewcommand{\theequation}{\thesection.\arabic{equation}}
\renewcommand{\thetheorem}{\thesection.\arabic{theorem}}
\renewcommand{\thelemma}{\thesection.\arabic{lemma}}
\newcommand{\bbr}{\mathbb R}
\newcommand{\bbs}{\mathbb S}
\newcommand{\bn}{\bf n}
\newcommand\charf {\mbox{{\text 1}\kern-.24em {\text l}}}
\newcommand\fg{{\mathfrak g}}
\newcommand\fk{{\mathfrak k}}
\newcommand\fp{{\mathfrak p}}
\newcommand\g{\gamma}
\newcommand\G{\Gamma}
\newcommand\ka{\kappa}
\newcommand\al{\alpha}
\newcommand\be{\beta}
\newcommand\lrt{\longrightarrow}
\newcommand\s{\sigma}
\newcommand\ba{\backslash}
\newcommand\lmt{\longmapsto}
\newcommand\CP{{\mathscr P}_n}
\newcommand\CM{{\mathcal M}}
\newcommand\BC{\mathbb C}
\newcommand\BZ{\mathbb Z}
\newcommand\BR{\Bbb R}
\newcommand\BQ{\mathbb Q}
\newcommand\Rmn{{\mathbb R}^{(m,n)}}
\newcommand\PR{{\mathcal P}_n\times {\mathbb R}^{(m,n)}}
\newcommand\Gnm{GL_{n,m}}
\newcommand\Gnz{GL_{n,m}({\mathbb Z})}
\newcommand\Gjnm{Sp_{n,m}}
\newcommand\Gnml{GL(n,{\mathbb R})\ltimes {\mathbb R}^{(m,n)}}
\newcommand\Snm{SL_{n,m}}
\newcommand\Snz{SL_{n,m}({\mathbb Z})}
\newcommand\Snml{SL(n,{\mathbb R})\ltimes {\mathbb R}^{(m,n)}}
\newcommand\Snzl{SL(n,{\mathbb Z})\ltimes {\mathbb Z}^{(m,n)}}
\newcommand\la{\lambda}
\newcommand\GZ{GL(n,{\mathbb Z})\ltimes {\mathbb Z}^{(m,n)}}
\newcommand\DPR{{\mathbb D}(\PR)}
\newcommand\Rnn{{\mathbb R}^{(n,n)}}
\newcommand\Yd{{{\partial}\over {\partial Y}}}
\newcommand\Vd{{{\partial}\over {\partial V}}}
\newcommand\Ys{Y^{\ast}}
\newcommand\Vs{V^{\ast}}
\newcommand\DGR{{\mathbb D}(\Gnm)}
\newcommand\DKR{{\mathbb D}_K(\Gnm)}
\newcommand\DKS{{\mathbb D}_{K_0}(\Snm)}
\newcommand\fa{{\frak a}}
\newcommand\fac{{\frak a}_c^{\ast}}
\newcommand\SPR{S{\mathcal P}_n\times \Rmn}
\newcommand\DSPR{{\mathbb D}(\SPR)}
\newcommand\BD{{\mathbb D}}
\newcommand\SP{{\mathfrak P}_n}
\newcommand\BH{{\mathbb H}}

\thanks{
\endgraf
Keywords and phrases\,: derivatives of $L$-functions, the Gross-Zagier formula, Shimura varieties, \\ \indent automorphic Green functions, Faltings heights, Colmez's conjecture.}

\begin{abstract} In this paper, we investigate the derivatives of $L$-functions, in particular,
the Riemann zeta function, the Hasse-Weil $L$-function, the Rankin $L$-function and the Artin $L$-function, and survey the relations between the derivatives of $L$-functions and the geometry and arithmetic of the associated Shimura varieties.
\end{abstract}
\maketitle


\vskip 7mm

\centerline{\large \bf Table of Contents}

\vskip 0.75cm $ \quad\qquad\textsf{\large \ 1.
Introduction}$\vskip 0.021cm

$\quad\qquad \textsf{\large\ 2. Derivatives of the Riemann zeta function}$
\vskip 0.0421cm

$ \quad\qquad  \textsf{\large\ 3. Derivatives of the Hasse-Weil $L$-function of an elliptic curve}$
\vskip 0.0421cm
\par
$\quad\qquad \qquad  \textsf{ 3.1.\ Basic\ notions}$
\par
$\quad\qquad \qquad  \textsf{ 3.2.\ Points\ of\ finite\ order}$
\par
$\quad\qquad \qquad  \textsf{ 3.3.\ The\ Mordell-Weil\ theorem}$
\par
$\quad\qquad \qquad  \textsf{ 3.4.\ The\ Birch-Swinnerton-Dyer\ conjecture}$
\par
$\quad\qquad \qquad  \textsf{ 3.5.\ Heegner\ points\ and\ Jacobi\ forms}$
\vskip 0.0421cm
$ \quad\qquad  \textsf{\large\ 4. The Gross-Zagier formula}$
\vskip 0.0421cm

$ \quad\qquad  \textsf{\large\ 5. Borcherds product}$
\vskip 0.0421cm

$ \quad\qquad \textsf{\large\ 6. Derivatives of the Rankin $L$-functions of the orthogonal}$ \\
\indent\ \ \ \ \ \ \ \ \ $\qquad\textsf{\large Shimura varieties of signature $(2,n)$}$
\vskip 0.0421cm

$ \quad\qquad  \textsf{\large\ 7. Derivatives of the Artin $L$-functions }$
\par
$\quad\qquad \qquad  \textsf{ 7.1.\ Faltings \ heights}$
\par
$\quad\qquad \qquad  \textsf{ 7.2.\ Colmez's\ conjecture}$
\vskip 0.0421cm

$ \quad\qquad
\textsf{\large\ 8. Final Remarks}$
\vskip 0.0421cm

$ \quad\qquad
\textsf{\large\ Appendix A. The L-function $L(F,U,s)$}$
\vskip 0.0421cm

$ \quad\qquad
\textsf{\large\ Appendix B. The Andr{\'e}-Oort Conjecture}$
\vskip 0.0421cm

$ \quad\qquad
\textsf{\large\ Appendix C. The Gross-Kohnen-Zagier theorem in higher dimension}$
\vskip 0.0421cm


$ \quad\qquad\ \textsf{\large References }$

\vskip 10mm


\begin{section}{{\bf Introduction}}
\setcounter{equation}{0}
\vskip 3mm
The theory of $L$-functions is one of the important and unifying themes of number theory and arithmetic geometry. $L$-functions are naturally and closely related to number fields, automorphic forms, Artin representations, Shimura varieties, abelian varieties, intersection theory etc.
In particular, central values of $L$-functions and their derivatives are curious and of interest because these values are closely related to geometric and arithmetic properties of the Shimura varieties. For example, the Gross-Zagier formula, Colmez's  conjecture, the averaged Colmez formula provide some evidences that there are some interesting and curious relations between geometrical properties and special values of derivatives of $L$-functions.

\vskip 3mm
The purpose of this article is to review the results about derivatives of $L$-functions obtained by various mathematicians, e.g., Gross, Zagier, the school of Kudla past three and a half decades. Here we will deal with derivatives of the Riemann zeta function, the Hasse-Weil $L$-function of an elliptic curve, the Rankin-Selberg convolution $L$-function associated to the orthogonal Shimura variety of signature $(n,2)$, and a certain Artin $L$-function.

\vskip 3mm
The paper is organized as follows.
In section 2, we deal with derivatives of the Riemann zeta function. We review the result obtained by Kudla, Rapoport and Yang \cite{KuRY, KuRY1}.
In section 3, we deal with derivatives of the
Hasse-Weil $L$-function of an elliptic curve. We briefly describes some materials about an elliptic curve \cite{Si} and then survey the topics (the rank of an elliptic curve and the Birch-Swinnerton-Dyer conjecture) related to central values of the derivative of the $L$-function of an elliptic curve.
Gross and Zagier \cite{GZ} described the height of Heegner points in terms of cental value of the first derivative of the $L$-function of an elliptic curve. Gross, Kohnen and Zagier \cite{GKZ} showed that the
heights of Heegner points are the coefficients of a modular form of weight $3/2$. Kolyvagin
\cite{Kol1, Kol2} used Heegner points to construct Euler systems, and used this to prove some of the Birch-Swinnerton-Dyer conjecture for elliptic curves of rank $1$.
In section 4, we review the Gross-Zagier formula obtained by Gross and Zagier \cite{GZ}. We mention that the Gross-Zagier formula was generalized to the Gross-Zagier formula on Shimura curves by the Zhang school \cite{Zh1, Zh2, YZZ}.
In section 5, we review Borcherds products \cite{Bo1, Bo2} which are needed to understand the topics of section 6 and section 7. This note was written on the base of Bruinier's lecture note \cite{Br}. Borcherds used his products to construct modular generating series with coefficients in the first Chow group ${\rm CH}^1$ for divisors on locally symmetric varieties associated to the orthogonal group $O(n,2)$.
In section 6, we survey the work of J. Bruinier and T. Yang \cite{BY2}. They related the Faltings height pairing of arithmetic special divisors and CM cycles on Shimura varieties associated to orthogonal groups of signature $(n,2)$ to the central values of certain Rankin-Selberg $L$-functions.
In section 7, we briefly survey the averaged Colmez formula obtained recently by Andreatta, Goren, Howard and Madapusi Pera \cite{AGHM2}. Almost at the same time the averaged Colmez formula was obtained independently by Yuan and Zhang \cite{YZ}. But so far Colmez's conjecture has not been solved. Using the averaged Colmez formula, J. Tsimerman \cite{T} proved the Andr{\'e}-Oort conjecture for the Siegel modular varieties.
In Appendix A, we describe the Rankin-Selberg $L$-function appearing in section 6 in some detail. We follow the papers \cite{BY2, BKY, Sch} closely.
In Appendix B, we briefly mention the Andr{\'e}-Oort conjecture and give a sketch of the proof of Tsimerman for the Andr{\'e}-Oort conjecture for Siegel modular varieties. This note is based on \cite{T, PT2}.
In Appendix C, we review the Gross-Kohnen-Zagier theorem in higher dimension obtained by Borcherds
\cite{Bo2}.

\vskip 0.31cm \noindent {\bf Notations:} \ \ We denote by
$\BQ,\,\BR$ and $\BC$ the field of rational numbers, the field of
real numbers and the field of complex numbers respectively. We
denote by $\BZ$ and $\BZ^+$ the ring of integers and the set of
all positive integers respectively. $\BQ^+$ (resp. $\BR^+$) denotes the set of all positive
rational (resp. real) numbers. We denotes by $\BZ_+$ (resp. $\BQ_+,\ \BR_+$) the set of all
non-negative integers (resp. rational numbers, real numbers). $\mathbb H$ denotes the Poincar{\'e}
upper half plane. $[a,b,c,d]$ denotes the $2\times 2$ matrix
$\begin{pmatrix}
   a & b \\
   c & d
 \end{pmatrix}$.
$\BQ^{\times}$ (resp. $\BR^{\times},\ \BC^{\times}$)
denotes the group of nonzero rational (resp. real, complex) numbers.
If $F$ is a number field, $F^{\times}$ denotes the multiplicative group of nonzero elements in $F$.
The symbol ``:='' means that
the expression on the right is the definition of that on the left.
For two positive integers $k$ and $l$, $F^{(k,l)}$ denotes the set
of all $k\times l$ matrices with entries in a commutative ring
$F$. For a square matrix $A\in F^{(k,k)}$ of degree $k$,
$\sigma(A)$ denotes the trace of $A$. For any $M\in F^{(k,l)},\
^t\!M$ denotes the transpose of $M$. For a positive integer $n$, $I_n$
denotes the identity matrix of degree $n$.
 For a complex matrix $A$,
${\overline A}$ denotes the complex {\it conjugate} of $A$.
${\rm diag}(a_1,\cdots,a_n)$ denotes the $n\times n$ diagonal matrix with diagonal entries
$a_1,\cdots,a_n$. For a smooth manifold, we denote by $C_c (X)$ (resp. $C_c^{\infty}(X)$ the algebra of all continuous (resp. infinitely differentiable) functions on $X$ with compact support. $\mathbb{H}$ denotes the Poincar{\'e} upper half plane.
\begin{equation*}
{\mathbb H}_g=\,\{\,\Omega\in \BC^{(g,g)}\,|\ \Omega=\,^t\Omega,\ \ \ \text{Im}\,\Omega>0\,\}
\end{equation*}
denotes the Siegel upper half plane of degree $g$.
If $z\in\BC,$ we put $e(z):=e^{2\pi iz}$. For an even lattice $L$ with a quadratic form, $D(L):=L'/L$
denotes the discriminant of $L$ where $L'$ is the dual lattice of $L$. We denote by $\BC [D(L)]$
the group algebra of the discriminant group $D(L)$. $H(\beta,m)$ denotes the Heegner divisor of
index $(\beta,m).$
For a regular scheme $\mathscr{X}$ which is projective and flat over ${\rm Spec}\,\BZ$, we denote by
$\widehat{\rm CH}^p(\mathscr{X})$ the arithmetic Chow group of codimension p.

\end{section}
\vskip 10mm

\newcommand\POB{ {{\partial}\over {\partial{\overline \Omega}}} }
\newcommand\PZB{ {{\partial}\over {\partial{\overline Z}}} }
\newcommand\PX{ {{\partial}\over{\partial X}} }
\newcommand\PY{ {{\partial}\over {\partial Y}} }
\newcommand\PU{ {{\partial}\over{\partial U}} }
\newcommand\PV{ {{\partial}\over{\partial V}} }
\newcommand\PO{ {{\partial}\over{\partial \Omega}} }
\newcommand\PZ{ {{\partial}\over{\partial Z}} }
\newcommand\PW{ {{\partial}\over{\partial W}} }
\newcommand\PWB{ {{\partial}\over {\partial{\overline W}}} }
\newcommand\OVW{\overline W}
\newcommand\Rg{{\mathfrak R}_n}

\begin{section}{{\bf Derivatives of the Riemann zeta function}}
\setcounter{equation}{0}
\vskip 3mm
The Riemann zeta function
\begin{equation*}
\zeta (s):= \sum_{n=1}^\infty \frac{1}{n^s},\quad s\in\BC
\end{equation*}
converges absolutely for ${\rm Re} (s)>1.$ It is well known that $\zeta (s)$ has the following properties\,:
\vskip 3mm\noindent
$(a)\ \zeta (s)$ has a meromorphic continuation to the whole complex plane with a simple pole at $s=1.$
Thus
\begin{equation}
\zeta (s)=\frac{1}{s-1}+a_0+a_1(s-1)+\cdots+a_k (s-1)^k+\cdots.
\end{equation}
\vskip 3mm\noindent
$(b)\ \zeta (s)$ has the functional equation
\begin{equation}
\Lambda (s):=\pi^{-s/2} \Gamma (s/2) \zeta(s)=\Lambda (1-s).
\end{equation}
\vskip 3mm\noindent
$(c)$ If $k\in\BZ^+$ is a positive even integer,
\begin{equation}
\zeta (k)=-(2\pi i)^k \cdot\frac{1}{2}\cdot\frac{B_k}{k!}.
\end{equation}

\vskip 3mm\noindent
$(d)\ \zeta (3)$ is irrational. This fact was proved by Roger Ap{\'e}ry (1916-1994) \cite{V}.
\vskip 3mm\noindent
$(e)$ In 1734, Leonard Euler (1707-1783) proved that
\begin{equation}
\frac{\zeta(2r)}{\pi^{2r}}\in\BQ \quad {\rm if}\ r\in\BZ^+.
\end{equation}

\vskip 3mm\noindent
$(f)$ Let
\begin{equation}
\gamma:=\lim_{n\lrt\infty} \left( \sum_{k=1}^n \frac{1}{k}-\log n \right)
\end{equation}
be the Euler-Mascheroni constant. Then $\gamma$ can be expressed as
\begin{equation}
\gamma=\sum_{k=2}^\infty (-1)^k \frac{\zeta(k)}{k}.
\end{equation}

\begin{remark}
In 1873, Charles Hermite (1823-1901) proved that $e$ is transcendental.
In 1882, Ferdinand von Lindemann (1852-1939) proved that $\pi$ is transcendental. But up to now we have no idea that $\gamma$ is irrational or transcendental. Three numbers $\pi,\,e$ and $\gamma$ are called the
${\sf Holy\ Trinity}$.
\end{remark}
\vskip 3mm
We have infinite product formulas containing the Holy Trinity $\pi,\,e$ and $\gamma$. First in 1997 H.\,S. Wilf \cite{Wi} obtained the following infinite product formula
\begin{equation}
\prod_{k=1}^\infty \left\{ e^{-1/k} \left( 1+\frac{1}{k}+\frac{1}{2\,k^2}\right)\right\}
=\frac{e^{\pi/2}+e^{-\pi/2}}{\pi\cdot e^{\gamma}}.
\end{equation}
Secondly J. Choi, J. Lee and H.\,M. Srivastava \cite{CS3} obtained the following infinite product formulas
generalizing the Wilf's formula (2.9)
\begin{equation}
\prod_{k=1}^\infty \left\{ e^{-1/k} \left( 1+\frac{1}{k}+\frac{\alpha^2+\frac{1}{4}}{k^2}\right)\right\}
=\frac{2(e^{\alpha\pi}+e^{-\alpha\pi})}{(4\alpha^2+1)\,\pi\cdot e^{\gamma}},\quad \alpha\in\BC,\
\alpha\neq \pm \frac{1}{2}i
\end{equation}
and
\begin{equation}
\prod_{k=1}^\infty \left\{ e^{-2/k} \left( 1+\frac{2}{k}+\frac{\beta^2+1}{k^2}\right)\right\}
=\frac{e^{\beta\pi}-e^{-\beta\pi}}{2\beta(\beta^2+1)\,\pi\cdot e^{2\gamma}},\quad \beta\in\BC,\
\alpha\neq \pm i.
\end{equation}

\begin{remark}
C.-P. Chen and R. Paris \cite{C} generalized the above two infinite products formulas (2.8) and (2.9).
\end{remark}

\begin{definition}
(1) The ${\sf Glaisher}$-${\sf Kinkelin\ constant}\ \mathcal A$ is defined to be
\begin{equation*}
\log \mathcal A :=\lim_{n\lrt \infty} \left\{ \left( \sum_{k=1}^n k\log k\right)-
\left( \frac{n^2}{2}+\frac{n}{2}+\frac{1}{12}\right)\log n +\frac{n^2}{4}\right\}.
\end{equation*}
We note that $\mathcal A \approx 1.282427130\cdots.$
\vskip 2mm\noindent
(2) The ${\sf Catalan\ constant}\ \mathscr G$ is defined to be
\begin{equation*}
  \mathscr G :=\sum_{k=0}^\infty \frac{(-1)^k}{(2k+1)^2} \approx 0.9159655\cdots.
\end{equation*}
\end{definition}

\begin{definition}
The ${\sf Hurwitz\ zeta\ function}\ \zeta(s,a)$ is defined by
\begin{equation*}
\zeta (s,a):=\sum_{k=0}^\infty (k+a)^{-s},\quad {\rm Re} (s)>1,\ a\neq -1,-2,\cdots.
\end{equation*}
\end{definition}

\vskip 3mm
Choi, Srivastava et al. obtained the following results involving the values of the derivatives of
$\zeta (s)$.

\begin{theorem}
\begin{equation}
  \mathcal A = \exp \left( -\zeta'(-1) + \frac{1}{12} \right),
\end{equation}
equivalently,
\begin{equation}
\zeta'(-1)=\lim_{n\lrt \infty} \left\{ -\sum_{k=1}^n k\log k +
\left( \frac{n^2}{2}+\frac{n}{2}+\frac{1}{12}\right)\log n -\frac{n^2}{4} + \frac{1}{12}\right\}.
\end{equation}
\end{theorem}
\begin{proof} The proof can be found in \cite{CS1}.\end{proof}

\begin{theorem}
\begin{equation}
  \zeta'(2)=\pi^2 \left\{ \frac{\gamma}{6}+ \frac{1}{6}\log (2\pi) -2\,\log \mathcal A \right\},
\end{equation}
equivalently,
\begin{equation}
\gamma=6\cdot \frac{\zeta'(2)}{\pi^2} + \log \frac{\mathcal A^{12}}{2\,\pi}.
\end{equation}
\end{theorem}
\begin{proof} The proof can be found in \cite[p.\,441]{CS2} or \cite[p.\,129]{C}.\end{proof}

\begin{theorem}
\begin{equation}
  \zeta'(-1,\frac{1}{4})=-\frac{1}{96} + \frac{1}{8} \log \mathcal A + \frac{\mathscr G}{4\pi}.
\end{equation}
\end{theorem}
\begin{proof} The proof can be found in \cite[p.\,440]{CS2}.\end{proof}

\vskip 3mm
Now we briefly describe the work of S. Kudla, M. Rapoport and T. Yang \cite{KuRY, KuRY1}. Let $B$ be an indefinite division algebra over $\BQ$ and let $\mathcal{O}_B$ be a maximal order in $B$. Let $D:=D(B)$ be the product of all primes at which $B$ is division. Let $\mathcal{M}$ be the moduli space of abelian surfaces with a special action of $\mathcal{O}_B$. Then $\mathcal{M}$ is an integral model of the Shimura curve attached to $B$. We may regard $\mathcal{M}$ as an arithmetic surface in the sense of `Arakelov theory \cite{F, Bos}.
The moduli stack $\mathcal{M}$ carries a universal abelian variety $\mathscr A /\mathcal{M}$, and the Hodge line bundle $\omega$ on $\mathcal{M}$ is defined
\begin{equation}
  \omega=\bigwedge^2 \left( {\rm Lie}(\mathscr A)\right)^*.
\end{equation}
We equip $\omega$ with the metric $\|\,\cdot\,\|$ which, for $z\in \mathcal{M} (\BC)$, is given by
\begin{equation}
\| \beta \|^2_z=e^{-2\,C}\,\frac{1}{4\pi^2}\,\int_{{\mathscr A}_z (\BC)} \beta \wedge \bar{\beta},
\end{equation}
where
\begin{equation}
C:=\frac{1}{2} \left( \log(4\pi) + \gamma \right),\qquad \gamma\ {\rm is\ the\ Euler\ constant}.
\end{equation}
Thus we obtain a class $\widehat{\omega}=(\omega,\|\,\cdot\,\|)\in \widehat{{\rm Pic}}(\mathcal{M}).$
We define the constant $\kappa$ by the relation
\begin{equation}
\frac{1}{2}\,{\rm deg}(\widehat{\omega})\cdot \kappa=\langle \widehat{\omega},\widehat{\omega} \rangle -\zeta_D(-1) \left[ 2 {{\zeta'(-1)}\over {\zeta(-1)}} +1-2\,C - \sum_{p| D}
{{p+1}\over {p-1}}\cdot \log (p) \right],
\end{equation}
where $\zeta_D (s):=\zeta (s) \prod_{p|D} (1-p^{-s}).$
Also writing $\widehat{\omega}$ for the image of $\widehat{\omega}$ in $\widehat{\rm CH}^1 (\mathcal{M})$, we set
\begin{equation}
\widehat{\mathcal Z} (0,v):=-\widehat{\omega}-(0,\log (v))+(0,\kappa) \in \widehat{\rm CH}^1 (\mathcal{M}),
\qquad v\in\BR\ {\rm with}\ v>0.
\end{equation}
\vskip 3mm
For each $m\in\BZ$ and for $v\in \BR$ with $v>0$, we define
\begin{equation}
\widehat{\mathcal Z}(m,v):=({\mathcal Z}(m), \Xi (m,v))\in \widehat{\rm CH}^1 (\mathcal{M}).
\end{equation}
Here, for $m=0,\ \widehat{\mathcal Z}(0,v)$ is defined by the formula (2.19),
for $m>0,\ {\mathcal Z}(m)$ is the divisor on $\mathcal{M}$ corresponding to those
$\mathcal O_B$-abelian surfaces which admit a special endomorphism $\alpha$ with $\alpha^2=-m,$ and for $m<0,\ {\mathcal Z}(m)=\emptyset$. For all $m\neq 0,\ \Xi(m,v)$ is the nonstandard Green's function introduced by Kudla \cite{Ku}.
\vskip 2mm
Using the Gillet-Soul{\'e} height pairing $\langle\,\, ,\,\,\rangle$ between $\widehat{\rm CH}^1 (\mathcal{M})$ and $\widehat{{\rm Pic}}(\mathcal{M})$ \cite{GS}, we form the height generating series
\begin{equation}
\phi_{\rm ht}(\tau):=\sum_{m\in\BZ} \langle \widehat{\mathcal Z}(m,y),\widehat{\omega} \rangle \,
e^{2\pi i m\tau}, \quad \tau=x+iy\in \mathbb H\ {\rm with}\ x,y\in\BR.
\end{equation}
The quantities $\langle \hat{\mathcal Z}(m,y),\widehat{\omega} \rangle$ can be thought of as arithmetic degres.

\vskip 3mm
Assume $D:=D(B)>1.$ Kudla, Rapoport and Yang \cite[Theorem A or Theorem 7.2]{KuRY} proved that
\begin{equation}
\phi_{\rm ht}(\tau)={\mathcal E}'(\tau,\frac{1}{2};D),
\end{equation}
where ${\mathcal E}(\tau,s;D)$ is the Eisenstein series of weight $\frac{3}{2}$. We refer to the formula (6.39) in \cite{KuRY} for the precise definition of ${\mathcal E}(\tau,s;D)$. Moreover, they proved that $\kappa=0$, i.e.,  that
\begin{equation}
\langle \widehat{\omega},\widehat{\omega} \rangle =\zeta_D(-1) \left[ 2 {{\zeta'(-1)}\over {\zeta(-1)}} +1-2\,C - \sum_{p| D}
{{p+1}\over {p-1}}\cdot \log (p) \right].
\end{equation}
We refer to \cite{KuRY} or \cite[Theorem 7.1.1]{KuRY1} for the proof of (2.23).
According to Theorem 2.2 and the formula (2.23), we obtain
\begin{equation}
{{\langle \widehat{\omega},\widehat{\omega} \rangle}\over {\zeta_D(-1)}}=  2 {{\zeta'(-1)}\over {\zeta(-1)}} +1 -\log (4\pi) - \frac{6}{\pi} \zeta'(2)-\log \left( \frac{\mathcal A^{12}}{2\pi}\right)- \sum_{p| D}{{p+1}\over {p-1}}\cdot \log (p).
\end{equation}

\end{section}

\vskip 10mm

\begin{section}{{\bf Derivatives of the Hasse-Weil $L$-function of an elliptic curve}}
\setcounter{equation}{0}

\vskip 0.35cm
In this section we review some basic materials about elliptic curves (cf.\,\cite{Si}) and some results about the derivatives of the Hasse-Weil $L$-function of an elliptic curve.

\vskip 0.5cm \noindent
{\bf 3.1. Basic notions}
\vskip 0.35cm
We shall call a nonsingular projective curve $E$ over a field $K$ of genus one an {\it elliptic curve}
over $K$. Then $E$ has a unique algebraic group structure with identity element $\infty$. The group law on the
set $E(K)$ of $K$-rational points (or more generally any extension of $K$) is defined by
$$P+Q+R=\infty,\qquad P,Q,R\in E(K)$$
if $P,Q,R$ are collinear. It is known that an elliptic curve $E$ can be embedded into the two dimensional projective
space ${\mathbb P}^2$ as a cubic curve defined by the following Weierstrass equation\,:
\begin{equation}
E\,:\quad y^2+a_1xy+a_3y=x^3+a_2x^2+a_4x+a_6,\quad a_1,\cdots,a_6\in K
\end{equation}
with its nonzero discriminant (which will be defined later)
$$\Delta_E=-b_2^2b_8-8b_4^3-27b_6^2+9b_2b_4b_6\neq 0,$$
where
$$b_2=a_1^2+4a_2,\quad b_4=2a_4+a_1a_3,\quad b_6=a_3^2+4a_6,$$
$$b_8=a_1^2a_6+4a_2a_6-a_1a_3a_4+a_2a_3^2-a_4^2={\frac 14}\left(b_2b_6-b_4^2\right).$$
We define the $j$-invariant $j(E)$ of $E$ by
\begin{equation}
j(E):={ {\left( b_2^2-24\,b_4\right)^3}\over {\Delta_E}}
\end{equation}
The invariant differential $\omega_E$ associated with the Weierstrass equation (3.1) is given by
\begin{equation}
\omega_E:={ {dx}\over{2y+a_1x+a_3}}={{dy}\over{3x^2+2a_2x+a_4-a_1y}}
\end{equation}

\vskip 0.2cm \noindent
{\bf Case I}. ${\rm char}(K)\neq 2,3$\,:
\vskip 0.1cm Using the transformation
$$x\mapsto u^2x'+r,\quad y\mapsto u^3y'+su^2x'r+6,$$
the equation (3.1) can be simplified to the following form
$$y^2=x^3+ax+b,\quad \Delta = -16 \left( 4a^3+27b^2\right)\neq 0.$$

\vskip 0.2cm \noindent
{\bf Case II}. ${\rm char}(K)= 2$\,:
\vskip 0.1cm
We see that $j(E)=0$ if and only if $a_1\neq 0$. If $a_1\neq 0$, choosing suitably $r,s,t,$ we can achieve
$a_1=1,\,a_3=0,\,a_4=0$ and the equation (3.1) takes the form
$$y^2+xy=x^3+a_2x^2+a_6$$
with the condition of smoothness given by $\Delta\neq 0.$ If $a_1=0\, (i.e.,\ j(E)=0),$ the equation (3.1)
transforms to
$$y^2+a_3y=x^3+a_4x+a_6$$
and the condition of smoothness in this case is $a_3\neq 0.$

\vskip 0.2cm \noindent
{\bf Case III}. ${\rm char}(K)= 3$\,:
\vskip 0.1cm
The equation (3.1) can be transformed to
$$y^2=x^3+a_2x^2+a_4x+a_6.$$
Here the multiple roots are disallowed.

\vskip 0.3cm
The proper Weierstrass form is
\begin{equation}
y^2=4x^3-g_2x-g_3
\end{equation}
The discriminant $\Delta$ of the equation (3.4) is given by
\begin{equation}
\Delta=g_2^3-27g_3^2\neq 0.
\end{equation}
The $j$-invariant $j$ of the equation (3.4) is given by
\begin{equation}
j=2^6\cdot 3^3\,{{g_2^3}\over {g_2^3-27g_3^2}}=1728\,{{g_2^3}\over {\Delta}}.
\end{equation}
We can see that two elliptic curves have the same $j$-invariant if and only if they are isomorphic over an algebraic closure
${\overline K}$ of $K$.

\vskip 0.3cm
Let
$${\mathbb H}=\{ \tau\in \BC\,|\ {\rm Im}(\tau)>0\ \}$$
be the Poincar{\`e} upper half plane. Let us fix $\tau\in {\mathbb H}$. We let
$$L_\tau:=\BZ+\BZ \tau=\{ m\tau+n\in\BC\,|\ m,n\in\BZ\,\}$$
be the lattice in $\BC$ with basic period $1$ and $\tau$. We put
\begin{equation}
g_2(\tau):=60 \sum_{\omega\in L_\tau \atop \omega\neq 0}{1\over {\omega^4}},\qquad
g_3(\tau):=140 \sum_{\omega\in L_\tau \atop \omega\neq 0}{1\over {\omega^6}}
\end{equation}
We recall the famous Weierstrass $\wp$-function defined by
\begin{equation}
{\wp}_\tau(z)={\wp}(z;L_\tau)={1\over{z^2}}+ \sum_{\omega\in L_\tau \atop \omega\neq 0}
\left\{ {1\over {(z-\omega)^2}} -{1\over {\omega^2}} \right\}.
\end{equation}
Then we get
\begin{equation}
{\wp}_\tau'(z)={{d{\wp}_\tau}\over {dz}}=-2 \sum_{\omega\in L_\tau }
 {1\over {(z-\omega)^3}}
\end{equation}
The field of elliptic functions with periods $L_\tau$ is generated over $\BC$ by ${\wp}$ and
${\wp}'$. Surprisingly we have the following equation
\begin{equation}
{\wp}_\tau'(z)^2=4{\wp}_\tau(z)^3-g_2(\tau){\wp}_\tau(z)-g_3(\tau).
\end{equation}
Let
\begin{equation}
E_\tau\,:\quad y^2=4x^3-g_2(\tau)x-g_3(\tau)
\end{equation}
be the elliptic curve. Then we see that the complex torus $\BC/L_\tau$ is biholomorphic to $E_\tau(\BC)$
via the map
\begin{equation}
f_\tau: \BC/L_\tau\lrt E_\tau(\BC),\qquad f_\tau(z)=[{\wp}_\tau(z):{\wp}_\tau'(z):1],\ z\in\BC.
\end{equation}
Conversely we consider the differential of the first kind
\begin{equation}
{{dx}\over y}={1\over {\sqrt{4x^3-g_2(\tau)x-g_3(\tau)}} }
\end{equation}
on the Riemann surface $E_\tau(\BC)$. We integrate this form over a path joining a fixed initial point (say $\infty$)
with a variable point. The integral depends on the choice of a path, but its image in $\BC/L_\tau$ is determined
only by the endpoints. More precisely
\begin{equation*}
z-z_0=\int_{{\wp}_\tau(z_0)}^{{\wp}_\tau(z)}{1\over {\sqrt{4w^3-g_2(\tau)w-g_3(\tau)}} }\,dw
\end{equation*}

\vskip 0.3cm
According to Jacobi, we can see that if we set $q:=e^{2\pi i\tau}$, the discriminant
$$\Delta_{E_\tau}=\Delta(\tau)=g_2(\tau)^3-27 g_3(\tau)^2$$
is expressed as
\begin{equation}
\Delta(\tau)=(2\pi)^{12}\,q\prod_{n=1}^\infty (1-q^n)^{24}=(2\pi)^{12}\,\sum_{n=1}^\infty \tau(n)\,q^n\qquad ({\rm Jacobi^{,}s\ identity}).
\end{equation}
The $j$-invariant $j(\tau)=j(E_\tau)$ of $E_\tau$ is given by
\begin{equation}
j(\tau)=1728\,{{g_3(\tau)^3}\over {\Delta(\tau)}}={1\over q}+744+196884\,q+ 21493760\,q^2+\cdots .
\end{equation}

Since $j(\tau)$ takes all complex values, every elliptic curve over $\BC$ is isomorphic to $E_\tau$ for some
$\tau\in {\mathbb H}$. We have the addition theorem for elliptic functions\,:
\begin{equation}
{\wp}_\tau(z_1+z_2)=-{\wp}_\tau(z_1)-{\wp}_\tau(z_2)+
{\frac 14} \left(  { {\wp}_\tau'(z_2)-{\wp}_\tau'(z_1) }\over
{ {\wp}_\tau(z_2)-{\wp}_\tau(z_1) } \right)^2,\qquad z_1,z_2\in \BC.
\end{equation}
The above formula induces the natural additive group structure on the elliptic curve (3.11). More precisely if
$P_1=(x_1,y_1),\ P_2=(x_2,y_2)$ are points of (2.11), the addition $P_3=(x_3,y_3)$ is given by
$$x_3=-x_1-x_2 + {\frac 14} \,\left( {y_2-y_1}\over {x_2-x_1}\right)^2.$$
It is easily seen that $f_\tau({\frac 12})=(\alpha_1,0),\ f_\tau({\frac {\tau}2})=(\alpha_2,0),\
f_\tau({\frac {1+\tau}2})=(\alpha_3,0),$ where $\alpha_1,\,\alpha_2,\,\alpha_3$ are the roots of the polynomial
$4x^3-g_2(\tau)x-g_3(\tau).$

\vskip 0.5cm \noindent
{\bf 3.2. Points of finite order}
\vskip 0.3cm
Let $E$ be an elliptic curve defined over a field $K$. For an integer $n\in\BZ$, we let $E[n]$ be the kernel of the
multiplication by $n$ defined by
\begin{equation}
[n]:E({\overline K})\lrt E({\overline K}), \quad [n](x):= nx,\ x\in E({\overline K}),
\end{equation}
where ${\overline K}$ is the algebraic closure of $K$. The isogeny $[n]$ has degree $n^2$. If $({\rm char}(K),n)=1$,
we have
\begin{equation}
E[n]\cong \BZ/n\BZ\times \BZ/n\BZ.
\end{equation}
However for ${\rm char}(K)=p$ and $n=p^m$, we have
\begin{equation}
E[n]\cong \left( \BZ/p^m\BZ\right)^{b_E},\quad b_E=0 \ {\rm or}\ 1.
\end{equation}
Assume that $({\rm char}(K),n)=1$. Let $K(E[n])$ be the Galois extension of $K$ generated by all the coordinates of all
points of $E[n]$. Then we have the Galois representation
\begin{equation}
\rho_n:{\rm Gal}({\overline K}/K)\lrt GL_2(\BZ/n\BZ)
\end{equation}
whose image is isomorphic to ${\rm Gal}(K(E[n])/K).$ We set $G_K={\rm Gal}({\overline K}/K)$. It is known that the representation
$\det (\rho_n)$ is the cyclotomic character of $G_K$ on the group $\mu_n$ of all $n$-th roots of unity. We observe
that $\mu_n\subset K(E[n]).$

\vskip 0.2cm
For any positive integer $k$ with $k\geq 2$, we put
\begin{equation}
G_{2k}(\tau):={\sum}_{m,n\in\BZ}' {1\over {(m\tau+n)^{2k}} },\quad \tau\in \BH,
\end{equation}
where the symbol $\sum'$ means that the summation runs over all pair of integers $(m,n)$ distinct from $(0,0)$. Then
$G_{2k}$ is a modular form of weight $2k$ with $G_{2k}(0)=2\,\zeta(2k).$ We put
\begin{equation}
E_{2k}(\tau)={{G_{2k}(\tau)}\over {2\zeta(2k)}}=
1-{{4k}\over {B_{2k}} }\sum_{n=1}^\infty \sigma_{2k-1}(n)q^n,
\end{equation}
where $\sigma_m(n)=\sum_{0<d|n}d^m$ is a divisor function and $B_k\,(k=0,1,2,\cdots)$ denotes the $k$-th Bernoulli number
defined by the formal power series expansion
$${x\over {e^x-1}}=\sum_{k=0}^\infty B_k\,{{x^k}\over {k!}}$$
Then $B_{2k+1}=0$ for all $k\geq 1$ and the first few $B_k$ are
$$B_0=1,\ B_1=-{\frac 12},\ B_2={\frac 16},\ B_4=-{\frac 1{30}},\ B_6={\frac 1{42}},\ B_8=-{\frac 1{30}},
\ B_{10}={\frac 5{66}},\cdots $$
We observe that
$$g_2(\tau)=60 G_4(\tau)={\frac 43}\,\pi^4 E_4(\tau)\quad {\rm and}\quad
g_3(\tau)=140 G_6(\tau)=\left({\frac 23}\right)^3 \pi^6\, E_6(\tau).$$

The Weierstrass equation for the elliptic curve
$$\BC/(2\pi i)L_\tau \stackrel{\thicksim}\lrt \BC^*/\langle q^\BZ\rangle,\qquad u\mapsto \exp (u)$$
is given by
$$Y^2=4X^3-{{E_4}\over {12}}X+{{E_6}\over {216}},$$
where
$$X= {\wp}(2\pi i u; 2\pi i L_\tau)\quad {\rm and}\quad Y={\wp}'(2\pi i u; 2\pi i L_\tau).$$
If we substitute
$$X=x+{1\over {12}}\quad {\rm and}\quad Y=x+2y,$$
we get a new equation of this curve with coefficients in $\BZ[[q]]$
$$T(q):\quad y^2=xy=x^3+A(q)x+B(q),$$
where
\begin{eqnarray*}
A(q)&=& -5 \sum_{n=1}^\infty \sigma_3(n)q^n=-5 \sum_{n=1}^\infty {{n^3q^n}\over {1-q^n}},\\
B(q)&=& {\frac 1{12}}\,\left\{ -5\left( {{E_4-1}\over {240}} \right)- 7\left( {{E_6-1}\over {-504}} \right) \right\}
=-{\frac 1{12}}\sum_{n=1}^\infty { {(7n^5+5n^3)q^n}\over {1-q^n} }
\end{eqnarray*}
This equation defines an elliptic curve over $\BZ((q))$ with the canonical differential $\omega_{can}$ given by
$${{dx}\over {2y+x}}={{dX}\over Y}$$
Let $N\in \BZ^+$ be a positive integer. Let us define
$$T\big(q^N \big):\quad y^2+xy=x^3+A\big( q^N \big)x+B\big( q^N\big).$$
We put $t=\exp(2\pi u).$ The points of order $N$ on $T(q^N)$ corresponding to $t=\zeta_N^iq^j\,(0\leq i,j\leq N-1)$ with
$\zeta_N=\exp(2\pi i/N),$ and their coordinates are given by
$$x(t)=\sum_{n\in\BZ} {{q^{Nn}t}\over {(1-q^{Nn}t)^2 } }-2 \sum_{n=1}^\infty
{ {nq^{Nn}}\over {1-q^{Nn}t } }$$
and
$$y(t)=\sum_{n\in\BZ} {{q^{2Nn}t^2}\over {(1-q^{Nn}t)^3 } }+ \sum_{n=1}^\infty
{ {nq^{Nn}}\over {1-q^{Nn}t } }$$

\vskip 0.5cm \noindent
{\bf 3.3. The Mordell-Weil theorem}
\vskip 0.3cm
\noindent
{\bf [The Mordell-Weil Theorem]} Let $E$ be an elliptic curve defined over a number field $K$. Then $E(K)$ is
finitely generated, that is,
$$E(K)\cong \BZ^{r_E} \oplus E(K)_{tor},$$
where $r_E$ is a nonnegative integer called the {\sf rank} of $E$ over $K$ and $E(K)_{tor}$ is the torsion subgroup of
$E(K)$.

\vskip 0.2cm\noindent
{\bf Remark.} When $K=\BQ$, Barry Mazur \cite{Ma1} showed that $E(K)_{tor}$ is isomorphic to one of the following groups\,:
$$\BZ/N\BZ\,(1\leq N\leq 10,\ N=12),\quad \BZ/2\BZ \times \BZ/2N\BZ\ (1\leq N \leq 4).$$
It is conjectured that there are elliptic curves of arbitrary large rank over $\BQ$. In 2006, Noam Elkies found an explicit elliptic $E$ over $\BQ$
curve with $r_E=28$ and described generators of $E(\BQ)$ explicitly.

\vskip 0.2cm The Shafarevich-Tate group ${\rm III}(E,K)$ is regarded as a cohomological obstruction to a calculation of $E(K)$.
It is conjectured that ${\rm III}(E,K)$ is finite. ${\rm III}(E,K)$ is defined as follows\,: We first consider the following exact sequence
\begin{equation}
0 \lrt E[n]\lrt E({\overline K})\stackrel{[n]}\lrt E({\overline K})\lrt 0.
\end{equation}
Let $G_K:=Gal({\overline K}/K)$. Then (3.23) yields an exact sequence of Galois cohomology groups
\begin{equation}
0 \lrt E(K)/[n](E(K)) \lrt H^1(G_K,E[n])\lrt H^1(G_K,E({\overline K}))[n]\lrt 0.
\end{equation}
For each place $v$ of $K$, we choose an extension $w$ of $v$ to ${\overline K}$ and denote by $G_v\subset G_K$
the corresponding decomposition subgroup $G_v\cong {\rm Gal}({\overline K}_w/K_v).$ Then we have the following commutative diagram

\begin{equation*}
\begin{array}{ccccccccc}
0 & \longrightarrow &  E(K)/[n](E(K))      & \longrightarrow & H^1(G_K,E[n])
 & \stackrel{\alpha}\longrightarrow & H^1(G_K,E({\overline K}))[n] &  \longrightarrow & 0\\
{} & {}             &\Big\downarrow              &  & {}
& & \ \ \Big\downarrow\,{\beta_n}{} & {}\\
0 & \longrightarrow & \prod_v E(K_v)/[n](E(K_v)) & \longrightarrow &
\prod_v H^1(G_v,E[n]) & \longrightarrow & \prod_v H^1(G_K,E({\overline K}_v))[n] &  \longrightarrow &
0,
\end{array}
\end{equation*}
where $\beta_n=\prod_v\beta_{n,v}$. Here $\beta_{n,v}$ denotes the composition of the restriction morphism and the morphism
induced by the inclusion $E({\overline K})\lrt  E({\overline K}_v).$ We define
the {\sf Shafarevich}\,-{\sf Tate group}
${\rm III}(E,K)$ of $E$ over $K$ to be
\begin{equation}
{\rm III}(E,K)=\bigcup_{n\in\BZ^+}{\rm III}(E,K)_n,\qquad {\rm III}(E,K)_n:={\rm Ker}(\beta_n).
\end{equation}
The group
\begin{equation}
S(E,K)_n:=\alpha^{-1}\big({\rm III}(E,K)_n\big)
\end{equation}
is called the {\sf Selmer group} of $E$. An element of $S(E,K)_n$ can be interpreted as the class of an $n$-covering $C\lrt E$ such that $C$ has a $K_v$-point in each completion $K_v$ of $K$. By definition we have an exact sequence
\begin{equation}
0 \lrt E(K)/[n](E(K)) \lrt S(E,K)_n\lrt {\rm III}(E,K)_n\lrt 0.
\end{equation}

\vskip 0.5cm \noindent
{\bf 3.4. The Birch-Swinnerton-Dyer conjecture}
\vskip 0.3cm
Let $E$ be an elliptic curve defined over a number field $K$. For a place $v$ of $K$ where $E$ has good reduction, we put
$$a({\mathfrak P}_v)=Nv+1- |E({\mathcal O}_k/{\mathfrak P}_v)|$$
and for places of $K$ with bad reductions, we set
$$a({\mathfrak P}_v)=
\begin{cases} \hfill\ \ 1 \ & \ {\rm if}\ v \ {\rm is\ split\ and \ of\ multiplicative\ reduction}\\
-1 & \ {\rm if}\ v \ {\rm is\ nonsplit\ and\ of\ multiplicative\ reduction}\\
\ \ 0 & \ {\rm if}\ E \ {\rm has\ additive\ reduction\ at}\ v.
\end{cases}$$
We define the $L$-function $L(E,s)$ of $E$ by
\begin{equation}
L(E,s)=\prod_{v:{\rm bad}}\Big( 1-a({\mathfrak P}_v)Nv^s\Big)^{-1}\cdot \prod_{v:{\rm good}}
\Big( 1-a({\mathfrak P}_v)Nv^s+Nv^{1-2s}\Big)^{-1}.
\end{equation}
Since $|a({\mathfrak P}_v)|_v\leq 2\sqrt{Nv}$, $L(E,s)$ converges absolutely for ${\rm Re}(s)>{\frac 32}.$
We assume the conjecture of the Hasse-Weil on the existence of an analytic continuation of $L(E,s)$ to the entire
complex plane. Let $r_E$ be the rank of $E$ and $a_E={\rm ord}_{s=1}L(E,s)$ the order of the zero of $L(E,s)$ at
$s=1$.

\vskip 0.5cm\noindent
{\bf [The Birch-Swinnerton-Dyer Conjecture]}
\vskip 0.2cm\noindent
(BSD1)\ $r_E=a_E.$
\vskip 0.2cm\noindent
(BSD2) Assume that the Shafarevich-Tate group ${\rm III}(E,K)$ is finite. Then
\begin{equation}
\lim_{s\rightarrow 1}{{L(E,s)}\over {(s-1)^{r_E}} }=M\,{ {|{\rm III}(E,K)|\,R_E}\over {|E(K)_{tor}|^2} }\, ,
\end{equation}
where $R_E$ is the elliptic regulator of $E$ and $M=\prod_{v\in S_E} m_v$ is an explicitly written product of local
Tamagawa factors over the set $S_E$ of all Archimedean places of $K$ and places where $E$ has bad reduction and
$m_v=\int_{E(K_v)}\omega,\ \omega$ being the N{\'e}ron differential of $E$.

\vskip 0.5cm
We consider the case $K=\BQ$. Then we have the functional equation
\begin{equation}
\Lambda(E,s):=\left( { {\sqrt{N}}\over {2\pi} }\right)^s \G(s)\, L(E,s)=\varepsilon (E)\, \Lambda(E,2-s),\quad \varepsilon (E)=\pm 1,
\end{equation}
where $\varepsilon (E)$ is the root number of $E$ and $N$ is the conductor of $E$. By the modularity of $E$
\cite{BCDT}, there is a primitive cusp form $f\in S_2^{new}(\G_0(N))$ of level $N$
\begin{equation}
L(E,s)=L(f,s).
\end{equation}
Let $\varphi:X_0(N)\lrt E$ be a modular parametrization of $E$. Then
$$\varphi^*\omega= 2\pi i f(\tau)d\tau\quad {\rm on}\ X_0(N)$$
and
$$L(E,1)= 2\pi \int_0^\infty f(iy)\,dy =m_\infty.$$

There are some evidences
supporting the BSD conjecture. We list these evidences
chronologically.
\vskip 0.2cm\noindent
{\bf Result
1\,}(Coates-Wiles\,\cite{CW},\,1977). Let $E$ be a CM curve over $\BQ$.
Suppose that $a_E$ is zero. Then $r_E$ is zero.
\vskip 0.2cm\noindent
{\bf Result
2\,}(Rubin\,\cite{R},\,1981). Let $E$ be a CM curve over $\BQ$. Assume
that $a_E$ is zero. Then the Tate-Shafarevich
group $\rm{III}(E,\BQ)$ of $E$ is finite.
\vskip 0.2cm\noindent
{\bf Result 3\,}(Gross-Zagier\,\cite{GZ},\,1986\,;\ \cite{BCDT},\,2001). Let $E$ be an
elliptic curve over $\BQ$. Assume that $a_E$ is
equal to one and $\varepsilon(E)=-1$. Then $r_E$ is equal to or bigger than one. The detail will be discussed.

\vskip 0.2cm\noindent
{\bf Result 4\,}(Gross-Zagier\,\cite{GZ},\,1986). There
exists an elliptic curve $E$ over $\BQ$ such that
$r_E=a_E=3$. For instance, the
elliptic curve ${\tilde E}$ given by
$${\tilde E}\ :\quad -139\,y^2=x^3+10\,x^2-20\,x+8$$
satisfies the above property.

\vskip 0.2cm\noindent
{\bf Result 5\,}(Kolyvagin\,\cite{Kol1, Kol2},\,
1990\,:\,Gross-Zagier\,\cite{GZ},\,1986\,:\,Bump-Friedberg-Hoffstein\,
\cite{BFH},
\,1990\,: Murty-Murty\,\cite{MM},\,1990\,:\,\cite{BCDT},\,2001). Let $E$ be
an elliptic curve over $\BQ$. Assume that $a_E$
is 1 and $\varepsilon(E)=-1$. Then $r_E$ is equal to 1.

\vskip 0.2cm\noindent
{\bf Result 6\,}(Kolyvagin\,\cite{Kol2},\,
1990\,:\,Gross-Zagier\,\cite{GZ},\,1986\,:\,Bump-Friedberg-Hoffstein\,\cite{BFH},
\,1990\,: Murty-Murty\,\cite{MM},\,1990\,:\,\cite{BCDT},\,2001). Let $E$ be
an elliptic curve over $\BQ$. Assume that $a_E$
is zero and $\varepsilon(E)=1$. Then $r_E$ is equal to
zero.
\vskip 3mm
Cassels proved the fact that if an elliptic curve over
$\BQ$ is isogeneous to another elliptic curve $E'$ over $\BQ$,
then the BSD conjecture holds for $E$ if and only if the BSD
conjecture holds for $E'$.

\vskip 0.5cm \noindent
{\bf 3.5. Heegner points and Jacobi forms}
\vskip 0.3cm
In this subsection, we describe the result of Gross-Kohnen-Zagier\,\cite{GKZ} roughly.
\vskip 0.2cm
First we begin with
giving the definition of Jacobi forms. By definition a Jacobi form
of weight $k$ and index $m$ is a holomorphic complex valued
function $\phi(z,w)\,(z\in \BH,\,z\in\BC)$ satisfying the
transformation formula
\begin{eqnarray}
\phi\left( {{az+b}\over {cz+d}},
{{w+\la z+\mu}\over{cz+d}}\right)=& e^{-2\pi i\left\{ cm(w+\la
z+\mu)^2 (cz+d)^{-1}-m(\la^2 z+2\la w)\right\}} \\ &\ \times
(cz+d)^k \,\phi(z,w) \nonumber
\end{eqnarray}
for all $\begin{pmatrix} a & b\\
c &d\end{pmatrix}\in SL(2,\BZ)$ and $(\la,\mu)\in\BZ^2$ having a
Fourier expansion of the form
\begin{equation}
\phi(z,w)=\sum_{\scriptstyle
n,r\in \BZ^2 \atop\scriptstyle r^2\leq 4mn} c(n,r)\,e^{2\pi
i(nz+rw)}.
\end{equation}
We remark that the Fourier coefficients
$c(n,r)$ depend only on the discriminant $D=r^2-4mn$ and the
residue $r\,(\text{mod}\ 2m).$ From now on, we put
$\Gamma_1\!:=SL(2,\BZ).$ We denote by $J_{k,m}(\Gamma_1)$ the space
of all Jacobi forms of weight $k$ and index $m$. It is known that
one has the following isomorphisms
\begin{equation}
[\Gamma_2,k]^M\cong
J_{k,1}(\G_1)\cong M_{k-{\frac 12}}^+(\G_0(4))\cong
[\G_1,2k-2],
\end{equation}
where $\G_2$ denotes the Siegel modular
group of degree 2, $[\G_2,k]^M$ denotes the Maass space introduced
by H. Maass\,(1911-1993)\,(cf.\,\cite{M1,M2,M3}), $M_{k-{\frac
12}}^+(\G_0(4))$ denotes the Kohnen space introduced by W.
Kohnen\,\cite{Koh} and $[\G_1,2k-2]$ denotes the space of modular forms
of weight $2k-2$ with respect to $\G_1$. We refer to \cite{Y1} and
\cite[pp.\,65--70]{Y3} for a brief detail. The above isomorphisms are
compatible with the action of the Hecke operators. Moreover,
according to the work of Skoruppa and Zagier \cite{SZ}, there is a
Hecke-equivariant correspondence between Jacobi forms of weight
$k$ and index $m$, and certain usual modular forms of weight
$2k-2$ on $\Gamma_0(N).$

\vskip 0.2cm
Now we give the definition of Heegner
points of an elliptic curve $E$ over $\BQ$. By \cite{BCDT}, $E$ is
modular and hence one has a surjective holomorphic map
$\phi_E:X_0(N)\lrt E(\BC).$ Let $K$ be an imaginary quadratic
field of discriminant $D$ such that every prime divisor $p$ of $N$
is split in $K$. Then it is easy to see that $(D,N)=1$ and $D$ is
congruent to a square $r^2$ modulo $4N$. Let $\Theta$ be the set
of all $z\in\BH$ satisfying the following conditions
$$az^2+bz+c=0,\quad a,b,c\in\BZ,\ N|a,$$
$$b\equiv
r\,\,(\text{mod}\,2N),\qquad D=b^2-4ac.$$ Then $\Theta$ is
invariant under the action of $\G_0(N)$ and $\Theta$ has only a
$h_K$ $\G_0(N)$-orbits, where $h_K$ is the class number of $K$.
Let $z_1,\cdots,z_{h_K}$ be the representatives for these
$\G_0(N)$-orbits. Then $\phi_E(z_1),\cdots,\phi_E(z_{h_K})$ are
defined over the Hilbert class field $H(K)$ of $K$, i.e., the
maximal everywhere unramified extension of $K$. We define the
Heegner point $P_{D,r}$ of $E$ by
\begin{equation}
P_{D,r}=\sum_{i=1}^{h_K}\phi_E(z_i).
\end{equation}
We observe that
$\varepsilon(E)=1$, then $P_{D,r}\in E(\BQ)$. Let $E^{(D)}$ be the
elliptic curve (twisted from $E$) given by
\begin{equation}
E^{(D)}\ :\ \ Dy^2=f(x).
\end{equation}
Then one knows that the $L$-series of $E$ over
$K$ is equal to $L(E,s)\,L(E^{(D)},s)$ and that $L(E^{(D)},s)$ is
the twist of $L(E,s)$ by the quadratic character of $K/\BQ$.

\vskip 0.3cm
\noindent
\begin{theorem}(Gross-Zagier\,\cite{GZ},\,1986\,;\,\cite{BCDT},\,2001). Let $E$ be an
elliptic curve over $\BQ$ of conductor $N$ such that $\varepsilon(E) =-1.$ Assume
that $D$ is odd. Then
\begin{equation}
L'(E,1)\,L(E^{(D)},1)=c_E\,u^{-2}\,|D|^{-{\frac 12}}\,{\hat
h}(P_{D,r}),
\end{equation}
where $c_E$ is a positive constant not
depending on $D$ and $r,\ u$ is a half of the number of units of
$K$ and ${\hat h}$ denotes the canonical N{\'e}ron-Tate height of $E$.
\end{theorem}
\vskip 0.2cm
Since
$E$ is modular by \cite{BCDT}, there is a cusp form of weight 2 with
respect to $\G_0(N)$ such that $L(f,s)=L(E,s).$ Let $\phi(z,w)$ be
the Jacobi form of weight 2 and index $N$ which corresponds to $f$
via the Skoruppa-Zagier correspondence. Then $\phi(z,w)$ has a
Fourier series of the form (3.33). B. Gross, W. Kohnen and D.
Zagier obtained the following result.

\vskip 0.2cm\noindent
\begin{theorem}(Gross-Kohnen-Zagier,\,\cite{GKZ}\,;\,\cite{BCDT},\,2001). Let $E$ be a
modular elliptic curve with conductor $N$ and suppose that
$\varepsilon=-1,\ r=1.$ Suppose that $(D_1,D_2)=1$ and $D_i\equiv
r_i^2\,(\text{mod}\,4N)\,(i=1,2).$ Then
\begin{equation*}
L'(E,1)\,c((r_1^2-D_1)/(4N),r_1)\,c((r_2^2-D_2)/(4N),r_2)\,=\,c_E'\, \langle P_{D_1,r_1},
P_{D_2,r_2} \rangle,
\end{equation*}
where $c_E'$ is a positive constant not depending
on $D_1,\,r_1$ and $D_2,\,r_2$ and $\langle\ ,\ \rangle$ is the height pairing
induced from the N{\' e}ron-Tate height function ${\hat h}$, that
is,  ${\hat h}(P_{D,r})=\langle P_{D,r},P_{D,r} \rangle$.
\end{theorem}
\vskip 0.2cm
We see from the
above theorem that the value of $\langle P_{D_1,r_1}, P_{D_2,r_2} \rangle$ of
two distinct Heegner points is related to the product of the
Fourier coefficients
$c((r_1^2-D_1)/(4N),r_1)$ \\
$c((r_2^2-D_2)/(4N),r_2)$ of the Jacobi
forms of weight 2 and index $N$ corresponded to the eigenform $f$
of weight 2 associated to an elliptic curve $E$. We refer to \cite{Y4}
and \cite{Z1} for more details.

\vskip 0.2cm\noindent
{\bf Corollary.} There is a point $P_0\in E(\BQ)\otimes_{\BZ}\BR$
such that
$$P_{D,r}=c((r^2-D)/(4N),r)P_0$$
for all $D$ and
$r\,(D\equiv r^2\,(\text{mod}\,4N))$ with $(D,2N)=1.$

\vskip 0.2cm
The corollary is obtained by combining Theorem 3.1 and Theorem 3.2 with
the Cauchy-Schwarz inequality in the case of equality.

\vskip 0.2cm\noindent {\bf Remark 4.} R. Borcherds \,\cite{Bo2} generalized the
Gross-Kohnen-Zagier theorem to certain more general quotients of
Hermitian symmetric spaces of higher dimension, namely to quotients
of the space associated to an orthogonal group of signature $(2,b)$
by the unit group of a lattice. Indeed he relates the Heegner
divisors on the given quotient space to the Fourier coefficients
of vector-valued holomorphic modular forms of weight $1+{\frac
b2}$.

\end{section}

\vskip 10mm
\begin{section}{{\bf The Gross-Zagier formula}}
\setcounter{equation}{0}
\vskip 3mm
In this section, we review the works of B. Gross and D. Zagier \cite{GZ} who solved the conjecture proposed by B. Birch early in the 1980s. These works had some applications to the study of the BSD conjecture and the class number problem of imaginary quadratic fields. They obtained the famous {\sf Gross-Zagier formula} relating the N{\'e}ron-Tate heights of Heegner points on the modular curve  to the central derivatives of some Rankin-Selberg $L$-functions under mild Heegner condition.

\vskip 0.3cm
Let $D$ be congruent to a square (mod $4N$). Assume that $D$ is odd and congruent to $1$ (mod 4). let $x=(E\lrt E')$
be a Heegner point of discriminant $D$ on $X_0(N)$. That is, the elliptic curves $E$ and $E'$ have complex multiplication by
${\mathcal O}={\mathcal O}_D$, where ${\mathcal O}$ is the ring of integers of the imaginary quadratic field
$K=\BQ(\sqrt{D})$ with $(D,N)=1.$ Let $h=h_K$ be the class number of $K$. Let $u=|{\mathcal O}^{\times}/\{\pm 1\}|$. Then
$$u=\begin{cases} 1\ &{\rm if}\ D\neq -3,-4\\
                  2\ &{\rm if}\ D=-4 \\
                  3\ &{\rm if}\ D=-3.
                  \end{cases}$$
Let $s$ be the number of distinct prime factors of $N$. It is known that there are $2^sh$ Heegner points on $X_0(N)$
that are all rational over the Hilbert class field $H=K(j(E))$ of $K$. For each place $v$ of $H$, we define
$|\ \cdot\ |_v:H_v^{\times}\lrt \BR_{\geq 0}$ by
$$|\alpha|_v=\begin{cases} \alpha {\overline \alpha}=|\alpha|^2\ \ & {\rm if}\ H_v\cong \BC\\
\ \ q_v^{-v(\alpha)} \ \ & {\rm if}\ H_v \ {\rm is\ nonarchimedean\ with\ prime}\ \pi\ {\rm satisfying} \\
      & \qquad v(\pi)=1\ {\rm and\ finite\ residue\ field\ of\ order}\ q_v
      \end{cases}$$
Let $J=J(X_0(N))$ be the Jacobian of $X_0(N)$. Then it is known that there is the canonical height pairing
\begin{equation}
\langle \ ,\ \rangle : J\times J \lrt \BR
\end{equation}
over $H$. The quadratic form
\begin{equation}
{\hat h}:J\lrt \BR\cup \{ 0\},\qquad {\hat h}(a)=\langle a,a\rangle,\quad a\in J
\end{equation}
associated to the height pairing $\langle \ ,\ \rangle$ is the canonical N{\'e}ron-Tate height associated to the class
of the divisor $(2\Theta)$, where $\Theta$ is a symmetric theta divisor in $J$. Since this divisor $\Theta$ is ample,
${\hat h}$ defines a positive definite quadratic form on $J(H)\otimes \BR$. This form can be extended to a
Hermitian form on $J(H)\otimes \BC$.

\vskip 0.2cm The Petersson inner product on $S_k(\Gamma_0(N))$ is defined by
\begin{equation}
(f_1,f_2):=\int_{\Gamma_0(N)\backslash \BH}f_1(\tau)\,{\overline{f_2(\tau)}}\, y^{k+2}\, dxdy,\quad \tau=x+i\,y,\
f_1,f_2\in S_k(\Gamma_0(N)).
\end{equation}
Let $f\in S_2^{new}(\G_0(N))$ whose Fourier expansion is $f(\tau)=\sum_{n\geq 1}a(n)q^n.$ Let $\sigma$ be a fixed element
in ${\rm Gal}(H/K)$ which is canonically isomorphic to the class group ${\rm Cl}_K$ of $K$ by the Artin map of global
class field theory. Let ${\mathfrak s}$ be the class in ${\rm Cl}_K$ corresponding to $\sigma$. We define
\begin{equation}
\theta_{\mathfrak s}(\tau)={\frac 1{2u}}+ \sum_{{\mathfrak a}\in {\mathfrak s}\atop {\mathfrak a}\ {\rm integral}}e^{2\pi i\,N{\mathfrak a}\,\tau}
=\sum_{n\geq 0}r_{\mathfrak s}(n)\,q^n,
\end{equation}
where $r_{\mathfrak s}(0)={\frac 1{2u}}$ and for $n\geq 1$,
\begin{equation*}
r_{\mathfrak s}(n)=|\{ {\mathfrak a}\in {\mathfrak s}\,|\ {\mathfrak a}\ {\rm integral},\ N{\mathfrak a}=n\,\}|.
\end{equation*}
Then $\theta_{\mathfrak s}(\tau)$ defines a modular form of weight $1$ on $\G_1(D)$ with a character
$\epsilon : \left( \BZ/D\BZ\right)^{\times}\rightarrow \{\pm 1\}$ associated with the quadratic extension $K/\BQ$.
To a new form $f\in S_2^{new}(\G_0(N))$ and the ideal class ${\mathfrak s}$ in ${\rm Cl}_K$ we associate the $L$-series
$L_{\mathfrak s}(f,s)$ defined by
\begin{equation}
L_{\mathfrak s}(f,s):=\sum_{n\geq 1\atop (n,DN)=1} {\frac {\epsilon(n)}{n^{2s-1}} }\cdot
\sum_{n\geq 1} {\frac {a(n)r_{\mathfrak s}(n)}{n^s} }
\end{equation}

\noindent
\begin{definition} Let $f\in S_2^{new}(\G_0(N))$ be a normalized eigenform of the Hecke algebra ${\mathbb T}$ and
$\chi: {\rm Cl}_K\rightarrow \BC^*$. We define the $L$-function
\begin{equation}
L(f,\chi,s):=
\sum_{{\mathfrak s}\in {\rm Cl}_K} \chi({\mathfrak s})\,L_{\mathfrak s}(f,s).
\end{equation}
\end{definition}

Both $L_{\mathfrak s}(f,s)$ and $L(f,\chi,s)$ converge absolutely in ${\rm Re}(s)> {\frac 32}.$ It is shown in \cite{GZ} that
they both have an analytic continuation to entire functions to the whole complex plane, satisfy the functional equations when $s$ is
replaced by $2-s$ and vanish at $s=1$.

\vskip 0.3cm We assume that $D$ is square-free and $(D,N)=1.$
\vskip 0.2cm \noindent
\begin{theorem}
Let $x=(E\lrt E')$ be a Heegner point of discriminant $D$. We put $c=(x)-(\infty)\in J(H).$ Let ${\rm Gal}(H/K)$ be an element corresponding
to ${\mathfrak s}\in {\rm Cl}_K$. Let $\langle\ ,\ \rangle$ be the canonical height pairing on $J(H)\otimes \BC$. Then the series
\begin{equation}
g_{\mathfrak s}(\tau)=\sum_{m\geq 1} \langle c,T_mc^\sigma \rangle\,q^m
\end{equation}
is an element of $S_2(\G_0(N))$ and satisfies the following property
\begin{equation}
(f,g_{\mathfrak s})={\frac {u^2 \sqrt{|D|}}{8\pi^2}}\,L_{\mathfrak s}'(f,1)
\end{equation}
for all $f\in S_2^{new}(\G_0(N))$.
\end{theorem}

\vskip 0.2cm \noindent
\begin{theorem}
Let $f\in S_2^{new}(\G_0(N))$ be a normalized eigenform of the Hecke algebra ${\mathbb T}$ and
$\chi: {\rm Cl}_K\rightarrow \BC^*$. Then
\begin{equation}
L'(f,\chi,1)={\frac {8\,\pi^2 (f,f)}{h\,u^2\sqrt{|D|}} }\,{\hat h}(c_{\chi,f}),
\end{equation}
where
$$c_{\chi}=\sum_{\s\in {\rm Gal}(H/K)}\chi^{-1}(\s) c^\s \in {\rm the}\ \chi\!-\!{\rm eigenspace\ of}\ J(H)\otimes\BC$$
and $c_{\chi,f}$ is the projection of $c_\chi$ to the $f$-isotypical component of $J(H)\otimes \BC.$
\end{theorem}
The proofs of the above two theorems can be found in \cite{GZ}.

\vskip 0.2cm Using Theorem 4.2 and a theorem of Waldspurger, Gross and Zagier \cite{GZ} obtained the following theorem:

\vskip 0.2cm \noindent
\begin{theorem}
Let $E$ be an elliptic curve over $\BQ$. Then there is a rational point $P$ in $E(\BQ)$ such that
$$L'(E,1)=\alpha\,\Omega\,\langle P,P\rangle,\quad \alpha\in \BQ^*,$$
where $\Omega$ is the real period of a regular differential on $E$ over $\BQ$. In particular,
\vskip 0.1cm (1) If $L'(E,1)\neq 0,$ there is an element of infinite order in $E(\BQ)$.
\vskip 0.2cm (2) If $L'(E,1)\neq 0$ and $r_E=1$, then
$$L'(E,1)=\beta\,\Omega\,R,\quad \beta\in\BQ^*,$$
where $R$ is the elliptic regulator of $E$.
\end{theorem}

\end{section}

\begin{remark}
The Gross-Zagier formula may be regarded as a generalization of the class number formula and the Kronecker
limit formula. The school of Zhang \cite{YZZ, Zh1, Zh2} generalized the Gross-Zagier formula on modular curves to the formula on quaternionic Shimura curves over totally real fields.
\end{remark}

\vskip 10mm
\begin{section}{{\bf Borcherds product}}
\setcounter{equation}{0}
\vskip 3mm
In this section, we review Borcherds products. We follow the notations in \cite{Br}.
\vskip 2mm
Let $(V,q)$ be a quadratic space of signature $(2,\ell)$ and $L\subset V$ an even lattice. We extend the bilinear form $(\cdot ,\cdot )$ attached to the quadratic form $q$ to a $\BC$-linear form on the complexification $V_\BC$ of $V$. Let $Gr (L)$ be the Grassmannian of $L$, that is, the real analytic manifold of two-dimensional positive definite subspaces of $V$. The subset
\begin{equation}
  {\mathscr K}:=\left\{ [z]\in \mathbb P (V_\BC)\,|\ (z,z)=0,\ \, (z,{\bar z})> 0\ \right\}
\end{equation}
has two connected components. We choose one of them, denoted by ${\mathscr K}^+$. We see easily that the identity component $O^+(V)$ of $O(V)$ acts on ${\mathscr K}^+$ transitively. Thus
${\mathscr K}^+\cong O^+(V) /H$ is a Hermitian symmetric space of dimension $\ell$, where $H$ is the stabilizer of a fixed point in ${\mathscr K}^+$.

\vskip 2mm
${\mathscr K}^+$ may be realized as a tube domain in $\BC^{\ell}$ as follows. Let $L'$ be the dual lattice
of $L$, and let $D(L):=L'/L$ be the discriminant group of $L$.
Suppose that $z\in L$ is a primitive norm $0$ vector (i.e., $\BQ z\cap L=\BZ z$ and $q(z)=0$)
and $z_*\in L'$ with $(z,z_*)=1.$ Then the sublattice
$K:=L\cap z^\perp \cap z_*^\perp$ is Lorenzian and
\begin{equation}
V=(K\otimes \BR)\oplus \BR z_* \oplus \BR z.
\end{equation}
If $z_L\in L\otimes \BC$ and $z_L=z_K +a z_* + bz$ with $z_*\in K\oplus \BC$ and $a,b\in \BC$, then we briefly write $z_L=(z_K,a,b).$ Let
\begin{equation}
{\mathscr H}_K:= \left\{ z_K=x_K + i\,y_K \in K\otimes\BC\,|\ x_K,\,y_K\in K\otimes \BR,\ y_K^2=(y_K,y_K)=0\,\right\}.
\end{equation}
We define the map $f_{K,L}:{\mathscr H}_K \lrt {\mathscr K}$ by
\begin{equation}
f_{K,L}(z_K):=[ (z_K,1,-q(z)-q(z_*))], \qquad z_K\in {\mathscr H}_K.
\end{equation}
Then we see easily that $f_{K,L}$ is biholomorphic. Let ${\mathscr H}_\ell$ be the component of
${\mathscr H}_K$ that is mapped to ${\mathscr K}^+$ under $f_{K,L}$. If we put the cone
$i\,C:={\mathscr H}_\ell \cap i\,(K\otimes \BR),$ then ${\mathscr H}_\ell=K\otimes \BR \oplus i\,C$.

\vskip 2mm
For $\beta\in D(L):=L'/L$ and $m\in\BZ+q(\beta)$ with $m<0$, we define a subset ${\tilde H}(\beta,m)$ of
$Gr (L)$ by
\begin{equation*}
{\tilde H}(\beta,m):=\bigcup_{\lambda\in \beta+L \atop q(\lambda)=m} \lambda^\perp.
\end{equation*}
Here $\lambda^\perp$ denotes the orthogonal complement of $\lambda$ in $Gr (L)$, that is,
the set of all positive definite two-dimensional subspaces $w\subset V$ with $w\perp \lambda$.
Let $\G$ be the orthogonal group $\G (L)$ or a subgroup of $\G (L)$ of finite index. Then
\begin{equation}
H(\beta,m):=\sum_{\lambda\in \beta+L \atop q(\lambda)=m} \lambda^\perp.
\end{equation}
is a $\Gamma$-invariant divisor on ${\mathscr H}_\ell$ with support ${\tilde H}(\beta,m)$. It is the inverse image under the canonical projection of an algebraic divisor on
$\Gamma \backslash {\mathscr H}_\ell$, which we call the {\sf Heegner divisor of index} $(\beta,m).$

\vskip 3mm
Let $z\in L$ be a primitive norm $0$ vector and $z_*\in L'$ with $(z,z_*)=1.$ Let $N$ be the unique positive integer such that $(z,L)=N\BZ$. We set $K:=L\cap z^\perp \cap z_*^\perp$.
If $v\in V=L\otimes\BR$, we write $v_K$ for the orthogonal projection of $v$ to $K\otimes\BR$. Then
\begin{equation*}
v_K=v-(v,z)z_*+(v,z)(z_*,z_*)z-(v,z_*)z.
\end{equation*}
Thus if $v\in L'$, then $v_K\in K'$. Let $\zeta\in L$ with $(\zeta,z)=N$. Then it can be represented uniquely as
\begin{equation*}
\zeta=\zeta_K +Nz_* +Bz,\qquad \zeta\in K',\ B\in \BQ.
\end{equation*}
It can be shown easily that $L$ can be written as
\begin{equation}
L=K\oplus \BZ\zeta \oplus \BZ z.
\end{equation}
It follows from (5.6) that $|D(L)|=N^2\, |D(K)|$. Moreover we obtain the isometric embedding
$\mathfrak{c}_{K',L'}:K'\lrt L'$ defined by
\begin{equation*}
\mathfrak{c}_{K',L'}(\gamma):=\gamma - \frac{(\gamma,\zeta)}{N}z,\qquad \gamma\in K'.
\end{equation*}
The kernel of the induced map $\overline{\mathfrak{c}}_{K',L'}:K'\lrt L'\lrt L'/L$ of $\mathfrak{c}_{K',L'}$ is equal to
$\{ \gamma\in K\,|\ (\gamma,\zeta)\equiv 0\,({\rm mod}\,N)\,\}.$

\vskip 3mm
We consider the sublattice of $L'$ defined by
\begin{equation*}
L'_0:=\left\{ \lambda\in L'\,|\ (\lambda,z)\equiv 0\,({\rm mod}\,N)\,\right\}.
\end{equation*}
Clearly $L\subset L'_0.$ The map $p:L'_0\lrt K'$ defined by
\begin{equation}
p(\lambda):=\lambda_K - \frac{(\lambda,z)}{N}\zeta_K,\qquad \lambda\in L'_0
\end{equation}
is a projection of $L'_0$ onto $K'$. It is easy to see that $p(L)=K$. Thus $p$ induces a surjective map ${\bar p}:L'_0/L \lrt K'/K$. We note that
\begin{equation*}
L'_0/L:=\left\{ \lambda\in L'/L\,|\ (\lambda,z)\equiv 0\,({\rm mod}\,N)\,\right\}.
\end{equation*}
Let $\beta\in D(L):=L'/L$ and $m\in\BZ+q(\beta)$ with $m<0$. We consider the theta integral
\begin{equation}
\Phi_{\beta,m}^L(v,s):=\int_{SL(2,\BZ)\backslash \BH} \langle F_{\beta,m}^L (\tau,s),\Theta_L(\tau,v)\rangle y
\frac{dx\,dy}{y^2},\quad v\in Gr(L),\ s\in \BC,
\end{equation}
where $F_{\beta,m}^L (\tau,s)$ denotes the Maass-Poincar{\'e} series of index $(\beta,m)$ with
$\tau\in\BH$ (cf.\,\cite[Definition 1.8,\,p.\,29]{Br}) and $\Theta_L(\tau,v):=\Theta_L(\tau,v;0,0)$
is the Siegel theta function (cf.\,\cite[(2.2),\,p.\,39]{Br}).

\vskip 3mm
The set of norm $1$ vectors in $V=L\otimes \BR$ has two component, one of them being given by
\begin{equation*}
V_1:=\left\{ v_1\in V\,|\ v_1^2=1,\ \,(z,v_1)>0\ \right\}.
\end{equation*}
We may identify $Gr(L)$ with via
\begin{equation*}
V_1 \lrt Gr (L),\qquad v_1 \mapsto \BR v_1.
\end{equation*}
This is the hyperboloid model of the hyperbolic space. Moreover $Gr (L)$ can be realized as the upper half space
\begin{equation*}
{\mathscr H}:=\left\{ (x_0,x_1,\cdots,x_{\ell-2})\in \BR^{\ell-1}\,|\ \,x_0 >0\ \right\}.
\end{equation*}
This is known as the upper half space model of hyperbolic space.

\vskip 3mm
In the coordinates of $V_1$ the Fourier expansion of $\Phi_{\beta,m}^L$ is given by
\begin{eqnarray}
  \Phi_{\beta,m}^L (v_1) &=& \frac{1}{\sqrt{2}\,(z,v_1)}\,\Phi_{\beta,m}^K + 4\,\sqrt{2}\,\pi\,(z,v_1)
  \sum_{\ell (N)}\,b(\ell z/N,0)\,\mathbb{B}_2 (\ell/N) \nonumber\\
   && +\, 4\,\sqrt{2}\,\pi\,(z,v_1) \sum_{\lambda\in p(\beta)+K \atop q(\lambda)=m}
   \mathbb{B}_2 \left( \frac{(\lambda,v_1)}{(z,v_1)}+(\beta,z_*) \right)\\
   && +\, 4\,\sqrt{2}\,(\pi/(z,v_1))^{-k}\,\sum_{\lambda\in  K'-0}
   \sum_{\delta\in L'_0/L \atop q(\delta)=\lambda+K}  b(\delta,q(\lambda))\,|\lambda|^{1-k} \nonumber \\
   &&\, \times \sum_{n\geq 1} n^{-k-1}\,e\left( n\,\frac{(\lambda,v_1)}{(z,v_1)}+ n\,(\delta,z')\right)\,
   K_{1-k}(2\pi n\,|\lambda|/(z,v_1)).\nonumber
\end{eqnarray}
Here $\Phi_{\beta,m}^K$ denotes the constant $\Phi_{\beta,m}^K (w,1-k/2)$, $b(\gamma,n)$ denotes the Fourier coefficients of the Poincar{\'e} series $F_{\beta,m}^L (\tau,1-k/2)$, $K_\nu$ is the modified Bessel function of the third kind, and
\begin{equation*}
\mathbb{B}_2 (x):=-2\, \sum_{n\in \BZ^\times} \frac{e(nx)}{(2\pi i n)^2}.
\end{equation*}

\begin{definition}
Using the above notations we define two functions $\xi_{\beta,m}^L$ and $\psi_{\beta,m}^L$ on $Gr(L)$ by
\begin{eqnarray}
  \xi_{\beta,m}^L (v_1) &=& \frac{\Phi_{\beta,m}^K}{\sqrt 2}\,
  \left( \frac{1}{(z,v_1)}-2\,(z_*,v_1) \right) + \frac{4\,\sqrt{2}}{(z,v_1)}\,
  \sum_{\lambda\in p(\beta)+K \atop q(\lambda)=m} (\lambda,v_1)^2 \nonumber  \\
&& +\, 4\,\sqrt{2}\,(\pi/(z,v_1))^{-k}\,\sum_{\lambda\in  K'-0}
   \sum_{\delta\in L'_0/L \atop q(\delta)=\lambda+K}  b(\delta,q(\lambda))\,|\lambda|^{1-k}  \\
   &&\, \times \sum_{n\geq 1} n^{-k-1}\,e\left( n\,\frac{(\lambda,v_1)}{(z,v_1)}+ n\,(\delta,z_*)\right)\,
   K_{1-k}(2\pi n\,|\lambda|/(z,v_1))\nonumber
\end{eqnarray}
and
\begin{equation*}
 \psi_{\beta,m}^L (v_1):=\Phi_{\beta,m}^L (v_1)-\xi_{\beta,m}^L (v_1).
\end{equation*}
\end{definition}
Note that this definition depends on the choice of the vector $z$ and $z_*$ and
$\xi_{\beta,m}^L =\xi_{-\beta,m}^L.$  It is known that $\Phi_{\beta,m}^L,\,\xi_{\beta,m}^L$ and
$\psi_{\beta,m}^L$ are real-valued functions (cf.\,\cite[Theorem 2.14]{Br}). Furthermore
according to \cite[Theorem 3.4]{Br},
$\psi_{\beta,m}^L$ is the restriction of a continuous piecewise linear function on $V$ and its only singularities lie on the Heegner divisor $H(\beta,m)$.

\begin{definition}
For $v\in Gr (L)-H(\beta,m)$, we define
\begin{equation*}
\Phi_{\beta,m}(v):={\rm CT}(\Phi_{\beta,m}(v,1-k/2)).
\end{equation*}
Here ${\rm CT}(\Phi_{\beta,m}(v,1-k/2))$ is the constant term of the Laurent expansion of
$\Phi_{\beta,m}(v,s)$ at $s=1-k/2.$ The set of all points $v\in Gr (L)-H(\beta,m)$ is not connected.
The connected components of $Gr (L)-H(\beta,m)$ are called the {\sf Weyl chamber} of $Gr(L)$ of
index $(\beta,m)$.
\end{definition}

\begin{definition}
Let $W$ be a Weyl chamber of index $(\beta,m)$. We define $\rho_{\beta,m}(W)$ to be the unique vector in $V$ satisfying the following property\,:
\begin{equation*}
\Psi_{\beta,m}^L(w):=8\,\sqrt{2\pi}\,(w,\rho_{\beta,m}(W)),\quad w\in W.
\end{equation*}
We call $\rho_{\beta,m}(W)$ the {\sf Weyl vector} of $W$. We refer to \cite[Definition 3.3]{Br} for the precise definition of $\Psi_{\beta,m}^L$.
\end{definition}

\begin{definition}
Let $f:\BH\lrt \BC \left[ D(L)\right]$ be a nearly holomorphic modular form for $Mp_2(\BZ)$ of weight $k=1-{\frac {\ell}2}$ with
its principal part
\begin{equation*}
\sum_{\gamma\in D(L)}\sum_{n\in \BZ+q(\g)\atop n<0} c(\g,n)\,{\mathfrak e}_\g (n\tau)
\end{equation*}
We call the components of
\begin{equation*}
Gr(K)-\bigcup_{\g\in L_0'/L} \bigcup_{n\in \BZ+q(\g)\atop n<0,\, c(\g,n)\neq 0} H(p(\g),n)
\end{equation*}
the {\sf Weyl\ chambers} of $Gr(K)$ with respect to $f$. Let $W$ be such a Weyl chamber. We define
\begin{equation*}
{\mathbb A}_L:=\left\{ (\g,n)\in L_0'/L\,| \ \g\in L_0'/L,\ n\in \BZ+q(\g)\ {\rm with}\ n<0,\ c(\g,n)\neq 0\,\right\}.
\end{equation*}
For any $(\g,n)\in {\mathbb A}_L,$ there is a Weyl chamber $W_{\g,n}$ of $Gr(K)$ of index
$({\bar p}(\g),n)$ such that $W\subset W_{\g,n}.$
Then
\begin{equation*}
W=\bigcap_{\g\in L'_0/L} \bigcap_{n\in \BZ+q(\g)\atop n<0,\, c(\g,n)\neq 0} W_{\g,n}=\bigcap_{(\g,n)\in {\mathbb A}_L}W_{\g,n}.
\end{equation*}
We define the {\sf Weyl\ vector} $\rho_f(W)$ attached to a Weyl chamber $W$ and $f$ by
\begin{equation}
\rho_f(W)={\frac 12}\sum_{\gamma\in L_0'/L}\sum_{n\in \BZ+q(\g)\atop n<0} c(\g,n)\,
\rho_{{\bar p}(\g),n}(W_{\g,n}).
\end{equation}

\end{definition}

R. Borcherds \cite{Bo1} proved the following

\begin{theorem} Let $L$ be an even lattice of signature $(2,\ell)$ with $\ell\geq 3$, and $z\in L$ a primitive isotropic vector. Let
$z_*\in L',\ K=L\cap z^\perp \cap z_*^\perp$ and $p:L'_0\lrt K'$ is the map defined by (5.7). Assume that $K$ contains an isotropic vector.
Let $f$ be a nearly holomorphic modular form of weight $k=1-{\frac {\ell}2}$ whose Fourier coefficients $c(\g,n)$ are {\sf integral} for
$n<0.$ Then the function
\begin{equation}
\Psi(Z)=\prod_{\beta\in D(L)}\prod_{m\in \BZ+q(\beta)\atop m<0} \Psi_{\beta,m}(Z)^{c(\beta,m)/2}
\end{equation}
is a meromorphic function on ${\mathscr H}_\ell$ with the following properties:
\vskip 0.2cm\noindent
{\sf (Bo1)} $\Psi(Z)$ is a meromorphic modular form of rational weight ${\frac {c(0,0)}2}$ for $\G (L)$ with some multiplier system $\chi$
of finite order. If $c(0,0) \in 2\BZ$, then $\chi$ is a character.
\vskip 0.2cm\noindent
{\sf (Bo2)} The divisor of $\Psi(Z)$ on ${\mathscr H}_\ell$ is given by
\begin{equation}
(\Psi)={\frac 12}\sum_{\beta\in D(L)}\sum_{m\in \BZ+q(\beta)\atop m<0} c(\beta,m) H(\beta,m).
\end{equation}
Let $m_{\beta,m}$ be the multiplicity of $H(\beta,m)$. Then
$$
m_{\beta,m}=\begin{cases} 2 \ & {\rm if}\ 2\beta=0\ {\rm in}\ L'/L \\
1 \ & {\rm if}\ 2\beta\neq 0 \ {\rm in}\ L'/L \end{cases}$$
We note that $c(\beta,m)=c(\beta,-m)$ and $H(\beta,m)=H(-\beta,m).$
\end{theorem}
\vskip 0.2cm\noindent
{\sf (Bo3)} Let $W\subset {\mathscr H}_\ell$ be a Weyl chamber of $Gr(K)$ with respect to $f$ and let
$$m_0={\rm min}\big\{ n\BQ\,|\ c(\g,n)\neq 0\
{\rm for\ all}\ \g\in L'/L\,\big\}.$$
On the set $U_\Psi$ of $Z\in {\mathscr H}_\ell$ which satisfy $q(Y)>|m_0|,$ and which belong to the complement of the set of
poles of $\Psi(Z)$, the function $\Psi(Z)$ has the Borcherds product expansion
\begin{equation}
\Psi(Z)=C\, e^{2\pi i \,(\rho_f(W),Z)}\prod_{\lambda\in K' \atop (\lambda, W)>0}
\prod_{\delta\in L_0'/L\atop p(\delta)=\lambda+K} \left( 1-e^{2\pi i\, ((\delta,z_*)+(\lambda,Z))}\right)^{c(\delta,q(\lambda))},
\end{equation}
where $C$ is a constant with $|C|=1$, and $\rho_f(W)\in K\otimes \BR$ is the Weyl vector attached to $W$ and $f$. Here for
$\lambda\in K'$ we write $(\lambda,W)>0$ if $(\lambda,w)>0$ for all $w$ in the interior of $W$.
\begin{proof}
The proof can be found in \cite[Theorem 13.3,\,\,pp.\,544--546]{Bo1} or \cite[Theorem 3.22,\,pp.\,88--91]{Br}.
\end{proof}
\begin{theorem}
Let $L$ be an even lattice of signature $(2,\ell)$, that splits two orthogonal hyperbolic planes over $\BZ$. Then every meromorphic modular form for $\G (L)$, whose divisor is a linear combination of Heegner divisors, is a Borcherds product in the sense of Theorem 5.1.
\end{theorem}
\begin{proof}
The proof can be found in \cite[Theorem 5.12,\,pp.\,139--140]{Br}.
\end{proof}

\end{section}

\vskip 10mm
\begin{section}{{\bf Derivatives of the Rankin $L$-functions associated to the orthogonal Shimura varieties}}
\setcounter{equation}{0}
\vskip 3mm
In this section we briefly review the relations between the Faltings height pairing of arithmetic special divisors and CM cycles on Shimura varieties associated to orthogonal groups of signature $(n,2)$, and  central derivatives of the Rankin $L$-functions associated with the above Shimura varieties obtained by J.-H. Bruinier and T. Yang \cite{BY2}.

\vskip 3mm
Let $(V,Q)$ be the quadratic space over $\BQ$ of signature $(n,2).$ Let $H:={\rm GSpin}(V)$ and
$G:=SL_2$, viewed as an algebraic group over $\BQ$. Its associated Hermitian symmetric space is
the Grassmannian
\begin{equation*}
  {\bf D}:=\left\{ z\subset V(\BR)\,|\ \dim z=2,\ \, Q|_z < 0 \,\right\}
\end{equation*}
of oriented negative definite two dimensional subspaces of $V(\BR)$. Let
\begin{equation*}
Mp(2,\BR):=\left\{ (g,\phi(\tau))\,|\ g=[a,b,c,d]\in SL(2,\BR),\ \,\phi(\tau)^2=c\tau+d, \tau\in
{\mathbb H} \,\right\}
\end{equation*}
be the metaplectic group with multiplication
\begin{equation*}
(g_1,\phi_1(\tau))\cdot (g_2,\phi_2(\tau))=(g_1g_2,\phi_1(g_2\tau)\phi_2(\tau)).
\end{equation*}
Let $L\subset V$ be an even lattice, and let $L'$ be the dual lattice. We denote by $S_L$ the subspace of Schwartz functions in the Schwartz space $S(V(\mathbb A_f))$ which are supported on
$L'\otimes {\widehat \BZ}$ and which are constant on cosets  of ${\widehat L}:=L\otimes {\widehat \BZ}$.
For any $\mu\in D(L):=L'/L$, the characteristic function $\phi_\mu:={\rm char}(\mu+{\widehat L})$ belongs to $S_L$ and
\begin{equation*}
S_L=\bigoplus_{\mu\in D(L)} \BC\,\phi_\mu \subset S(V(\mathbb A_f)).
\end{equation*}
Let $\rho_L:Mp(2,\BZ)\lrt {\rm Aut}(S_L)$ be the Weil representation on $S_L$ defined by
\begin{equation}
\rho_L (T)(\phi_\mu):=e^{2\pi i\, (\mu,\mu)}\phi_\mu
\end{equation}
and
\begin{equation}
\rho_L (S)(\phi_\mu):=\frac{e^{\frac{2-n}{4}}}{\sqrt{|D(L)|}}\,\sum_{\nu\in D(L)} e^{-2\pi i \,(\mu,\nu)}\phi_\nu,
\end{equation}
where
\begin{equation*}
T:=([1,1,0,1],1)\quad {\rm and}\quad S:=([0,-1,1,0],\sqrt{\tau})
\end{equation*}
are the standard generators of $Mp(2,\BZ).$
\vskip 2mm
Assume $k\leq 1$. A $C^2$-function $f:{\mathbb H}\lrt S_L$ is called a {\sf harmonic\ weak\ Maass\ form} of weight $k$ with respect to $Mp(2,\BZ)$ and $\rho_L$ if it satisfies the conditions (HM1)-(HM3):
\vskip 2mm\noindent
${\sf (HM1)} \ \ f(\gamma\cdot \tau)=\phi(\tau)^{2k} \rho_L({\tilde \gamma}) f(\tau),\ \,
\tau\in {\mathbb H}$\ \ for all ${\tilde \gamma}=(\gamma,\phi(\tau))\in Mp(2,\BZ).$
\vskip 2mm\noindent
${\sf (HM2)}$ \,There is a $S_L$-valued Fourier polynomial
\begin{equation}
P_f (\tau)=\sum_{\mu\in D(L)}\sum_{m\leq 0} c^+ (m,\mu)\, q^m \,\phi_\mu
\end{equation}
\ \ \ \ \ \ \ such that $f(\tau)-P_f (\tau)= O(e^{-\varepsilon y})$ as $y\lrt \infty$ for some $\varepsilon >0.$
\vskip 2mm\noindent
${\sf (HM3)}\ \ \Delta_k f=0$, where
\begin{equation}
\Delta_k:= -y^2 \left( \frac{\partial^2}{\partial x^2} + \frac{\partial^2}{\partial y^2} \right)
+ i\, k\, y \left( \frac{\partial}{\partial x} + i\,\frac{\partial}{\partial y} \right),\quad \tau=x+iy\in {\mathbb H}.
\end{equation}
We denote by $H(k,\rho_L)$ the vector space of all harmonic weak Maass forms of weight $k$
with respect to $Mp(2,\BZ)$ and $\rho_L$. Any $f\in H(k,\rho_L)$ has a unique decomposition $f=f^+ +f^-$, where
\begin{equation}
f^+ (\tau):=\sum_{\mu\in D(L)}\left(  \sum_{ m\in\BQ,\,  m\gg -\infty }
c^+ (m,\mu) \,q^m \right)\,\phi_\mu,\quad ({\rm the\ holomorphic\ part \ of}\ f)
\end{equation}
and
\begin{equation}
f^- (\tau):=\sum_{\mu\in D(L)}\left(  \sum_{m\in\BQ,\,  m<0}
c^- (m,\mu) \,\Gamma (1-2k,2\pi |m| y)\,q^m \right)\,\phi_\mu.
\end{equation}
\indent \indent \indent \indent \indent \indent\indent \indent (the non-holomorphic part of $f$)
\vskip 2mm\noindent
Here $\Gamma (s)$ is the usual Gamma function.

\vskip 3mm
We define the Maass lowering operator $L_k$ and the Maass raising operator $R_k$ by
\begin{equation}
L_k:=-2\,i\,y^2 \frac{\partial}{\partial {\overline \tau}},
\end{equation}
\begin{equation}
R_k:=-2\,i\, \frac{\partial}{\partial  \tau}+k\,y^{-1},
\end{equation}
and the Bruinier-Funke anti-linear differential operator \cite{BF}
\begin{equation*}
\xi_k: H(k,\rho_L) \lrt S_{2-k,{\bar \rho}_L}
\end{equation*}
by
\begin{equation}
f(\tau) \mapsto \xi_k (f)(\tau):=y^{k-2} \overline{L_k f(\tau)}.
\end{equation}
Here $S_{2-k,{\bar \rho}_L}$ denotes the vector space of all cusp forms of weight $2-k$ with respect to $Mp(2,\BZ)$ and ${\bar \rho}_L$. We see easily that the kernel of $\xi_k$ is the space of weakly holomorphic modular forms of weight $k$ with respect to $Mp(2,\BZ)$ and $\rho_L$, denoted by $M_{k,\rho_L}^{!}$.
We have the following exact sequence \cite[Corollary 3.8]{BF}
\begin{equation}
0\lrt M_{k,\rho_L}^{!} \lrt H(k,\rho_L) \lrt S_{2-k,{\bar \rho}_L} \lrt 0.
\end{equation}
There is a bilinear pairing between $M_{2-k,{\bar \rho}_L}$ and $H(k,\rho_L)$ defined by the Petersson scalar product
\begin{equation*}
\{ g,f\}:=(g,\xi_k (f))_{\rm Pet},\quad g\in M_{2-k,{\bar \rho}_L},\ f\in H(k,\rho_L).
\end{equation*}
It follows from the exactness of the sequence (6.10) that the induced pairing between
$S_{2-k,{\bar \rho}_L}$ and $H(k,\rho_L)/M_{k,\rho_L}^{!}$ is non-degenerate.

\vskip 3mm
Let $(V,Q),\ L,\ H={\rm GSpin}(V)\cdots$ be as above. Let $K\subset H(\mathbb A_f)$ be a compact open subgroup acting trivially on $S_L$. The Shimura variety
\begin{equation}
X_K:=H(\BQ)\backslash ({\bf D}\times H(\mathbb A_f))/K
\end{equation}
is a quasi-projective variety of dimension $n$. We consider the special divisors on $X_K$.
Let $x\in V(\BQ)$ be a vector of positive norm, and let $V_x$ be the orthogonal complement of $x$ in $V$. If $H_x$ is the stabilizer of $x$ in $H$, then $H_x\cong {\rm GSpin}(V_x)$. The sub-Grassmannian
\begin{equation}
{\bf D}_x:=\left\{ z\in {\bf D}\,|\ z\perp x \ \right\}
\end{equation}
defines an analytic divisor on ${\bf D}$. For $h\in H(\mathbb A_f)$, we consider the natural map
\begin{equation}
H_x(\BQ)\backslash ({\bf D}_x \times H_x(\mathbb A_f))/ (H_x(\mathbb A_f)\cap hKh^{-1}) \lrt X_K,
\quad (z,h_1)\mapsto (z,h_1 h).
\end{equation}
Its image defines a divisor $Z(x,h)$ on $X_K$ which is rational over $\BQ$. For $m\in \BQ^+$, we let
\begin{equation*}
\Omega_m:=\{ x\in V\,|\ Q(x)=m\,\}.
\end{equation*}
If $\Omega_m$ is not empty, then by Witt's theorem, we have $\Omega_m (\BQ)=H(\BQ)x_0$ and
$\Omega_m (\mathbb A_f)=H(\mathbb A_f)x_0$ for a fixed element $x_0\in \Omega_m (\BQ)$.

\vskip 2mm
For $\varphi \in S_L \subset S(\mathbb A_f),$ we may write
\begin{equation*}
{\rm supp}(\varphi)\cap \Omega_m (\mathbb A_f)=\bigsqcup_j K \alpha_j^{-1}x_0 \quad
{\rm a\ finite\ disjoint\ union},\ \alpha_j\in H(\mathbb A_f)
\end{equation*}
because ${\rm supp}(\varphi)$ is compact and $\Omega_m (\mathbb A_f)$ is a closed subset of
$V (\mathbb A_f)$. We set, for any $m\in \BQ^+$ and the above $\varphi\in S_L$,
\begin{equation}
Z(m,\varphi):=\sum_j \varphi (\alpha_j^{-1}x_0)\,Z(x_0,\alpha_j).
\end{equation}
We observe that $Z(m,\varphi)$ is a divisor on $X_K$ independent of the choice of $x_0$ and the representatives $\alpha_j$. We write, for $\mu\in D(L)$,
\begin{equation}
Z(m,\mu):=Z(m,\varphi_\mu).
\end{equation}
\indent
Let $f\in H(1-\frac{n}{2},{\bar \rho}_L).$ We consider the {\sf regularized\ theta\ integral}
\begin{equation}
\Phi (z,h,f):=\int_{\mathcal F}^{\rm reg} \langle f(\tau), \theta_L (\tau,z,h)\rangle d\mu (\tau),
\end{equation}
where $z\in {\bf D},\ h\in H(\mathbb A_f),\
{\mathcal F}=\{ \tau=x+iy\in \mathbb H\ |\ |x| \leq  \frac{1}{2},\ |\tau|\leq 1\,\}$ and $\theta_L$ is a $S_L$-valued theta function. We refer to \cite[pp.\,638-639]{BY2} or [Appendix A. (A.7)]
for the precise definition of $\theta_L$.
This integral is regularized, i.e., $\Phi (z,h,f)$ is defined as the constant term in the Laurent expansion at $s=0$ of the function
\begin{equation}
\lim_{T\lrt \infty}\int_{\mathcal F_T} \langle f(\tau), \theta_L (\tau,z,h)\rangle\, y^{-s}d\mu (\tau),
\end{equation}
where ${\mathcal F}_T:=\{ \tau=x+iy\in {\mathcal F}\,|\ y\leq T\,\}$. Then according to \cite[Theorem 4.2, pp.\,650-651]{BY2}, we know that $\Phi (z,h,f)$ has the following properties (R1)--(R4)\,:

\vskip 2mm\noindent
{\sf (R1)} $\Phi (z,h,f)$ is smooth on $X_K \smallsetminus Z(f)$, where
\begin{equation*}
Z(f):=\sum_{\mu\in D(L)}\sum_{m>0} c^+ (-m,\mu)\,Z(m,\mu).
\end{equation*}
{\sf (R2)} $\Phi (z,h,f)$ has a logarithmic singularity along the divisor $-2\,Z(f).$
\vskip 2mm\noindent
{\sf (R3)} $dd^c \Phi (z,h,f)$ can be extended to a smooth $(1,1)$-form on $X_K$. We have the \\
 \indent \indent Green current equation
\begin{equation*}
dd^c [\Phi (z,h,f)]+\delta_{Z(f)}=[dd^c \Phi (z,h,f)],
\end{equation*}
\indent \indent where $\delta_{Z(f)}$ denotes the Dirac current of a divisor $Z(f)$.

\vskip 2mm\noindent
{\sf (R4)} $\Delta_z \Phi (z,h,f)=\frac{n}{4}\cdot c^+ (0,0),$ where $\Delta_z$ is the invariant Laplacian on ${\bf D}$, normalized \\
\indent \indent as in \cite{Br}.

\vskip 3mm\noindent
It follows from (R1)--(R4) that $\Phi (z,h,f)$ is a Green function for $Z(f)$ in the sense of Arakelov geometry in the normalization of \cite{Sou}.

\vskip 3mm
Now we consider the CM $0$-cycles on $X_K$.
Let $U\subset V$ be a negative definite two-dimensional $\BQ$-subspace of $V$. It determines a two points subset $\{ z_U^{\pm} \} \subset {\bf D}$ given by $U(\BR)$ with the possible choices of orientation. Let $V_+\subset V$ be the orthogonal complement of $U$ over $\BQ$. Then we have the rational splitting
\begin{equation*}
V=V_+ \oplus U.
\end{equation*}
Let $T:={\rm GSpin}(U)\subset H$ and put $K_T:=K\cap T({\mathbb A}_f).$ Then we obtain the CM cycle
\begin{equation}
Z(U):=T(\BQ)\backslash \left( \{ z_U^{\pm} \}\times T({\mathbb A}_f)\right)/K_T \lrt X_K.
\end{equation}
Here each point in the cycle is counted with multiplicity $\frac{2}{w_{K,T}}$ with
$w_{K,T}:=|T(\BQ)\cap K_T|.$ Let
\begin{equation}
N:=L\cap U,\qquad S:=L\cap V_+
\end{equation}
be two definite lattices. For any $g\in S_{1+\frac{n}{2},\rho_L}$, we define the $L$-function by means of the convolution integral
\begin{equation}
L(g,U,s):=\langle \theta_S (\tau)\otimes E_N(\tau,s;1), g(\tau) \rangle_{\rm Pet},
\end{equation}
where $E_N(\tau,s;l)$ is a $S_L$-valued Eisenstein series of weight $l$ (cf.\,\cite[(2.16),\,p.\,641]{BY2} or [Appendix A.\,(A.12)]).
Then $L(g,U,s)$ has a meromorphic continuation to the whole complex plane because of the meromorphic continuation of $E_N(\tau,s;1)$ to $\BC$. We note that $L(g,U,0)=0$
because $E_N(\tau,s;l)$ is incoherent. When
$N \cong (\mathfrak A, -\frac{N}{N(\mathfrak A)})$ for a fractional ideal $\mathfrak A$ of
$k=\BQ(\sqrt D)$, we put
\begin{equation}
L^*(g,U,s):=\Lambda (\chi_D,s+1) L(g,U,s),
\end{equation}
where
\begin{equation}
\Lambda (\chi_D,s) :=|D|^{\frac{s}{2}}\Gamma_{\BR} (s+1) L(\chi_D,s),\qquad \Gamma_{\BR}(s):=\pi^{-\frac{s}{2}}\,\Gamma\left(\frac{s}{2}\right).
\end{equation}
Then we obtain the functional equation
\begin{equation}
L^*(g,U,s):=- L^*(g,U,-s).
\end{equation}
Under the assumption that $\Phi(z,h,f)$ is harmonic, i.e., $\Delta_z \Phi(z,h,f)=0$, Bruinier and Yang
\cite[Theorem 4.7, pp.\,654-655]{BY2} computed the value of $\Phi(z,h,f)$ at the CM cycle $Z(U)$\,:
\begin{equation}
\Phi (Z(U),f)=\deg (Z(U))\cdot \left( {\rm CT}(\langle f^+ (\tau), \theta_P (\tau)\otimes
{\mathscr E}_N (\tau) \rangle \right) + L'(\xi (f),U,0),
\end{equation}
where ${\rm CT}(\,\cdot\,)$ denotes the constant term of a Fourier series, $f^+ (\tau)$ is given by (6.5)
and ${\mathscr E}_N (\tau)$ is the holomorphic part of the derivative $E_N'(\tau,0;1)$ at the Eisenstein series associated to the lattice $N=L\cap U$ (cf.\,\cite[(2.26),\, p.\,643]{BY2} or [Appendix A.\,(A.20)]).

\begin{definition}
Let $\mathfrak X \lrt {\rm Spec}\, \BZ$ be an arithmetic variety of relative dimension $n$.
Let $Z^d (\mathfrak X)$ be the group of codimension $d$ cycles on $\mathfrak X$. If $F\subset \BC$
is a subfield, we put
\begin{equation*}
{\widehat {\rm CH}}^1(\mathfrak X)_F :={\widehat {\rm CH}}^1(\mathfrak X)\otimes_\BZ F.
\end{equation*}
Then there is a height pairing
\begin{equation*}
\langle \,\, ,\,\, \rangle_{\rm Fal} : {\widehat {\rm CH}}^1(\mathfrak X) \times Z^n (\mathfrak X)\lrt \BR
\end{equation*}
defined by
\begin{equation}
\langle {\widehat x}, y \rangle_{\rm Fal} : = \langle  x, y \rangle_{\rm fin} +
\langle {\widehat x}, y \rangle_{\infty},\quad {\widehat x}=(x,g_x)\in {\widehat {\rm CH}}^1(\mathfrak X),
\ y\in Z^n (\mathfrak X),
\end{equation}
where $\langle  x, y \rangle_{\rm fin}$ denotes the intersection pairing at the finite places and
\begin{equation*}
\langle {\widehat x}, y \rangle_{\infty}:=
\frac{1}{2}\, g_x (y(\BC)).
\end{equation*}
The quantity $\langle {\widehat x}, y \rangle_{\rm Fal}$ is called the {\sf Faltings\ height} of $y$ with
respect to ${\widehat x}$.
\end{definition}

Assume that there is a regular scheme ${\mathfrak X}_K \lrt {\rm Spec}\, \BZ$, projective and flat over $\BZ$, whose associated complex variety is a smooth compactification $X_K^*$
of $X_K$. Let $\mathscr Z (m,\mu)$ and $\mathscr Z (U)$ be suitable extensions to ${\mathfrak X}_K$
of $Z(m,\mu)$ and $Z(U)$ respectively. For any $f\in H(1-\frac{n}{2} ,{\overline \rho}_L )$, we set
\begin{equation}
\mathscr Z (f):= \sum_{\mu\in D(L)} \sum_{m>0} c^+ (-m,\mu)\,\mathscr Z (m,\mu)
\end{equation}
and
\begin{equation}
\widehat{\mathscr Z} (f):= (\mathscr Z (f), \Phi (\,\cdot \,,f)) \in {\widehat {\rm CH}}^1
(\mathfrak X_K)_\BC .
\end{equation}
We note that
\begin{equation}
\langle \widehat{\mathscr Z} (f), \mathscr Z (U) \rangle)_\infty = \frac{1}{2} \,
\Phi ( Z(U),f).
\end{equation}

\vskip 3mm
Bruinier and Yang \cite{BY2} formulated the following conjecture.
\vskip 3mm \noindent
{\bf Conjecture.}
For $f\in H(1-\frac{n}{2} ,{\overline \rho}_L )$, one has
\begin{equation}
\langle \widehat{\mathscr Z} (f), \mathscr Z (U) \rangle_{\rm Fal}= \frac{\deg (Z(U))}{2}
\left( c^+ (0,0)\,\kappa (0,0) +\,L'(\xi (f),U,0) \right),
\end{equation}
where $c^+ (0,0)$ is given by (6.5) and $\kappa (0,0)$ is the $(0,0)$-th coefficient of
\begin{equation*}
\mathcal E_N (\tau):= \sum_{\mu\in D(N)} \sum_{m\in \BQ} \kappa (m,\mu)\, q^m \,\phi_\mu,
\quad D(N):=N'/N.
\end{equation*}
We refer to the formula (A.19) in Appendix A for the precise definition of $\kappa (m,\mu).$

\vskip 3mm
Bruinier and Yang proved that the above conjecture holds for the case $n=0$
\cite[Theorem 6.5, p.\,663]{BY2} and also for the case $n=1$ \cite[Theorem 7.14, p.\,676]{BY2}.
The proof of the case $n=1$ gives another proof of the famous Gross-Zagier formula \cite{GZ}.
They also proved that the above conjecture holds for very special cases when $n=2.$

\begin{remark}
Bruinier, Howard and Yang \cite{BHY} studied special cycles on integral models of Shimura varieties associated to unitary similitude groups of signature $(n-1,1)$. They constructed an arithmetic theta lift from harmonic weak Maass forms of weight $2-n$ to the arithmetic Chow group
$\widehat{\rm CH}^1 ({\mathcal M}_{\mathbb L}^*)$ of the integral model of a unitary Shimura variety, and then they proved in loc.cit. that the height pairing of the arithmetic theta lift of
$f\in H_{2-n}(\omega_{\mathbb L})^{\Delta}$ against a CM cycle is equal to the central value of the derivative of the Rankin-Selberg convolution $L$-function of a cusp form $\xi (f)$ of weight $n$ and the theta function of a positive definite Hermitian lattice of rank $n-1$. When specified to the case $n=2$, this result can be viewed as a variant of the Gross-Zagier formula for Shimura curves associated to unitary groups of signature $(1,1)$.
\end{remark}

\end{section}

\vskip 10mm
\begin{section}{{\bf Derivatives of the Artin $L$-functions}}
\setcounter{equation}{0}
\vskip 3mm
In this section we describe the relations between the Faltings heights of CM abelian varieties and the derivatives of certain Artin $L$-functions. An interesting relation, so-called the {\sf averaged\ Colmez's\ formula} was proved recently by Andreatta, Goren, Howard and Madapushi-Pera (simply [AGHM]) \cite{AGHM2}, and independently by Yuan and Zhang (simply [YZ]) \cite{YZ}. Using the averaged Colmez's formula,
J. Tsimerman \cite{T} proved the Andr{\'e}-Oort conjecture for the Siegel modular variety $\mathcal A_g\ (g\geq 1)$.

\vskip 0.5cm \noindent
{\bf 7.1. Faltings heights}
\vskip 0.3cm
Let $A$ be an abelian variety of dimension $d$ over a number field $K$. Without loss of generality we may assume that $A$ has semistable reduction over $K$. Let $\mathcal A$ be the N{\'e}ron model of $A$ over the ring of integers $\mathcal O_K$.  Let
\begin{equation*}
  \omega_{\mathcal A}:=\pi_* \Omega^d_{\mathcal A / \mathcal O_K}
\end{equation*}
be the Hodge bundle of $A$ over ${\rm Spec}\,\mathcal O_K$, where
$\pi:\mathcal A \lrt {\rm Spec}\,\mathcal O_K$ is the structure morphism. Then there exists a canonical Hermitian metric $\|\,\, ,\,\|=\left\{ \|\,\,,\,\|_\sigma \right\}_\sigma$ on $\omega_{\mathcal A}$ given by
\begin{equation*}
\|\alpha\|^2_\sigma:=\frac{1}{2^d} \Big| \int_{A_\sigma (\BC)}\alpha\wedge {\overline \alpha}\Big|,
\quad A_\sigma (\BC):=\sigma (A)(\BC),
\end{equation*}
where $\sigma:K\lrt \BC$ is any embedding and
$\alpha\in \omega_{\mathcal A}\otimes_\sigma \BC = {\rm Hom} (A_\sigma(\BC), \Omega^d_{\mathcal A_\sigma (\BC)/ \BC})$ is any nonzero global holomorphic $d$-form on ${\mathcal A}_\sigma (\BC)$. Moreover we may assume that the class of $\omega_{\mathcal A}$ is trivial. So we may write
$\omega_{\mathcal A}=\mathcal O_K \alpha$ for some nonzero
$\alpha\in H^0({\rm Spec}\,\mathcal O_K,\mathcal O_K)=H^0(\mathcal A,\Omega^d_{\mathcal A / \mathcal O_K}).$ If $d=1,\ \alpha$ is the usual N{\'e}ron differential over an elliptic curve. The pair
$(\overline{\omega}_{\mathcal A},\|\,\,,\,\,\|)$ is a hermitian line bundle over ${\rm Spec}\,\mathcal O_K$.
The ${\sf Faltings\ height}$ of $A$ is defined to be
\begin{eqnarray}
  h_{\rm Fal}(A) &=& \frac{1}{[K:\BQ ]}\,\widehat{\deg} (\overline{\omega}_{\mathcal A}) \\
   &=& \frac{1}{2[ K:\BQ ]}\,\sum_{\sigma:K\lrt \BC}\log \left( \frac{1}{2^d}\,
   \Big| \int_{A_\sigma (\BC)}\alpha\wedge \overline{\alpha}\Big| \right),\nonumber
\end{eqnarray}
where $\widehat{\deg} (\overline{\omega}_{\mathcal A})$ denotes the arithmetic degree of
$\overline{\omega}_{\mathcal A}$\,\cite[p.\,67]{Sou}.

\vskip 0.5cm \noindent
{\bf 7.2. Colmez's conjecture}
\vskip 0.3cm
Let $E$ be a CM field of degree $2d$ with the maximal totally real subfield $F$ and let $c:E\lrt E$ be the conjugation. Let $A$ be an abelian variety over $\BC$ of dimension $d$ with complex multiplication by the maximal order $\mathcal O_E \subset E$ and having CM type $\Phi \subset {\rm Hom}(E,\BC)$. By the theory of complex multiplication, there is a number field $K$ such that $A$ is defined over $K$ and has a smooth projective integral model $\mathcal A$ over ${\rm Spec}\,\mathcal O_K$. P. Colmez \cite[Theorem II.\,2.10 (ii),\, p.\,665]{Col} proved that the Faltings height $h_{\rm Fal}(A)$ of $A$ depends only on the CM type $(E,\Phi)$, and not on $A$ itself. We denote it by
\begin{equation}
  h_{\rm Fal}(E,\Phi):=h_{\rm Fal}(A).
\end{equation}
Moreover, he formulated a conjectural formula for $h_{\rm Fal}(E,\Phi)$ in terms of the logarithmic derivatives at $s=0$ of certain Artin $L$-functions constructed in terms of the purely Galois-theoretic input $(E,\Phi)$ (cf. \cite[Conjecture II. 2.11,\,p.\,665]{Col}). The precise conjecture may be described as follows\,:
\vskip 3mm
\noindent
{\bf Colmez's\ conjecture.} {\it Let $E,\,\Phi,\,A,\cdots$ be as above. Let $E^{\rm nc}$ be the normal closure of $E$. Then we have the following identity\,:
\begin{eqnarray}
  h_{\rm Fal}(E,\Phi)&=&\sum_{\rho} c_{\rho,\Phi} \left( \frac{L'(0,\rho)}{L(0,\rho)} +
  \frac{\log f_\rho}{2} \right) \\
  &=& \sum_{\rho} c_{\rho,\Phi} \left( \frac{d}{ds}\Big|_{s=0} \log L(s,\rho) + \frac{\log f_\rho}{2} \right), \nonumber
\end{eqnarray}
where $\rho$ runs over irreducible complex representations of the Galois group ${\rm Gal}(E^{\rm nc}/E)$
for which $L(0,\rho)$ does not vanish, $L(s,\rho)$ is the Artin $L$-function determined by $\rho$, $c_{\rho,\Phi}$ are rational numbers depending only on the finite combinatorial data given by $\Phi$ and ${\rm Gal}(E^{\rm nc}/E)$, and $f_\rho$ is the Artin conductor of $\rho$.}

\vskip 5mm\noindent
\begin{remark}
Colmez's conjecture is still open in general. When $d=1,\ E$ is a quadratic imaginary field, and thus Colmez's conjecture is a form of the Chowla-Selberg formula. When $E/\BQ$ is an abelian extension, Colmez proved his conjecture in loc. cit., up to a rational multiple of $\log 2$. This extra error term was subsequently removed by A.\,Obus \cite{O}. When $d=2$, T.\,Yang \cite{Ya2} was able to prove Colmez's conjecture in many cases, including the first known cases of non-abelian extensions.
\end{remark}

\vskip 3mm
Recently, the {\sf averaged version\ of\ Colmez's\ conjecture} was proved by [AGHM], and independently by [YZ]. The averaged Colmez's conjecture  can be stated in the following way\,:

\begin{theorem}
The averaged Colmez's conjecture holds if we one averages over all CM types, up to a small error. Precisely,
\begin{equation}
  \frac{1}{2^d} \sum_{\Phi} h_{\rm Fal}(E,\Phi)=-\frac{1}{2}\cdot\frac{d}{ds}\Big|_{s=0} \log L(s,\chi)-\frac{1}{4}
  \log \Big| \frac{d_E}{d_F}\Big| - \frac{d}{2}\cdot\log (2\pi),
\end{equation}
where $\chi:\mathbb{A}_F^{\times}\lrt \{ \pm 1 \}$ is the quadratic Hecke character determined by the extension $E/F$, and $L(s,\chi)$ is the Artin $L$-function without the local factors at archimedian places.
The sum on the left hand side ranges over all $2^d$ CM types of $E$, and $d_E$\,(resp.\,$d_F$) is the discriminant of $E$\,(resp.\,$F$).
\end{theorem}

\begin{remark}
The proof of the above theorem given by [AGHM] \cite{AGHM2} is quite diferent from that given by [YZ] \cite{YZ}. The proof of [AGHM] is based on the computation of arithmetic intersection numbers of high dimensional Shimura varieties of orthogonal type, and the idea of T. Yang \cite{Ya2} on the Clomez's conjecture for $d=2.$ On the other hand, the proof given by [YZ] \cite{YZ} is based on the work of
Yuan-Zhang-Zhang \cite{YZZ} on the generalized Gross-Zagier formula and the computation of intersection numbers over Shimura curves.
\end{remark}

\vskip 3mm
Here we will give a rough sketch of the main ideas of the proof given by
[AGHM]. Let $F$ be a totally real field of degree $d$, and a quadratic space $(\mathscr V, \mathscr Q)$
over $F$ of dimension $2$ and signature $((0,2),(2,0),\cdots,(2,0)).$ In other words, $\mathscr V$ is negative definite at one archimedian place and positive definite at the rest. The even Clifford algebra $E:={\rm Cl}^+ (\mathcal V)$ is a CM field of degree $2\,d$ with $F$ as its maximal totally real subfield.
We define the quadratic space
\begin{equation}
(V,Q):=(\mathscr V,{\rm Tr}_{F/\BQ}\circ \mathscr Q)
\end{equation}
over $\BQ$ of signature $(n,2)$ with $n=2\,(d-1).$ We put $H:={\rm GSpin}(V)$. Fix a maximal lattice $L\subset V.$ Let $L'$ be the dual lattice of $L$ with respect to the symmetric bilinear form
\begin{equation*}
  [x,y]:=Q(x,y)-Q(x)-Q(y).
\end{equation*}
The associated Hermitian space
\begin{equation*}
\mathbb{D}=\left\{ z\in V_\BC\,|\ \,[z,z]=0,\ \, [z,\overline{z}] < 0 \right\}/\BC^{\times}.
\end{equation*}
Let $K$ be a compact open subgroup of $H(\mathbb{A}_f).$ The GSpin Shimura variety
\begin{equation*}
M(\BC):=H(\BQ)\backslash (\mathbb{D}\times H(\mathbb{A}_f )/K
\end{equation*}
is the space of complex points of a smooth algebraic stack $M$ over $\BQ$.
Let $T$ be the torus over $\BQ$ with points
\begin{equation*}
T(\BQ)=E^{\times}/{\rm ker}\,({\rm Nm}: F^{\times}\lrt \BQ^{\times}).
\end{equation*}
The relation (7.5) induces a morphism $T\lrt H$ so that one can construct a zero dimensional Shimura variety $Y$ over $E$, together with a morphism $Y\lrt M$ of $\BQ$-stacks. The image of this morphism consists of special points in the sense of Deligne and are the ${\sf big\ CM\ points}$ of \cite{BKY}.
[AGHM] define an integral model ${\mathcal Y}$ of $Y$, regular and flat over ${\mathcal O}_E$ along with a morphism $\mathcal Y \lrt \mathcal M$ of $\BZ$-stacks. The composition of $\widehat{\rm deg}$ and the pullback
\begin{equation*}
\widehat{\rm Pic}(\mathcal M) \lrt \widehat{\rm Pic}(\mathcal Y)
\stackrel{\widehat{\deg}}{\lrt} \BR
\end{equation*}
defines a linear functional,called the ${\sf arithmetic\ degree\ along}\ \mathcal Y$ denoted by
\begin{equation*}
\widehat{\mathscr Z} \longmapsto  [\widehat{\mathscr Z}:\mathcal Y ].
\end{equation*}

The completed $L$-function
\begin{equation}
\Lambda (s,\chi):=\Big| \frac{d_E}{d_F} \Big|^{s/2} \Gamma_{\BR}(s+2)^d \,L(s,\chi)
\end{equation}
satisfies the functional equation $\Lambda (1-s,\chi)=\Lambda (s,\chi)$ and also
\begin{equation}
\frac{\Lambda' (0,\chi)}{\Lambda (0,\chi)}=\frac{L' (0,\chi)}{L (0,\chi)}
+ \frac{1}{2}\,\log \Big| \frac{d_E}{d_F}\Big|  - \frac{d}{2}\, \log (4\pi e^\gamma).
\end{equation}
Here $\gamma$ denotes the Euler constant and $\Gamma_{\BR}(s)$ is defined by (6.21).

\vskip 3mm
For a fixed Schwartz function $\varphi\in S(\mathscr{V})$, we let $E(g,s,\varphi)$ be the corresponding
{\sf incoherent} weight $1$ Eisenstein series on $SL_2(\mathbb{A}_F)$. For any
$\vec{\tau}=(\tau_1,\cdots,\tau_d)\in \mathbb{H}^d,$
let $g_{\vec{\tau}}\in SL_2(\mathbb{A}_F)$ be the matrix with archimedean components
\begin{equation*}
g_{\tau_j}=
\begin{pmatrix}
  1 & x_j \\
  0 & 1
\end{pmatrix}
\begin{pmatrix}
  y_j^{1/2} & 0 \\
  0 & y_j^{-1/2}
\end{pmatrix}, \quad \tau_j=x_j+i y_j\in \mathbb{H}\ (1\leq j\leq d)
\end{equation*}
and take all finite components to be the identity matrix. The Hilbert modular Eisenstein series of
weight $1$
\begin{equation*}
E(\vec{\tau},s,\varphi):=N(\vec{y})^{-1/2}\cdot E(g_{\vec{\tau}},s,\varphi),\quad
N(\vec{y}):=y_1\cdots y_d
\end{equation*}
was defined in \cite[(4.4)]{BKY}.
\vskip 3mm
Let $f\in H(2-d,\rho_L)$ with {\sf integral} principal part. Then $\xi (f)\in S_{d,\overline{\rho}_L}$
has a decomposition $\xi(f)=\sum_{\mu\in D(L)} \xi(f)_\mu (\tau)\,\varphi_\mu.$ Here $\xi$ denotes the Bruinier-Funke differential operator defined by (6.9).
Bruinier-Kudla-Yang \cite[(5.3)]{BKY} defined the following $L$-function
\begin{equation}
\mathscr{L}(s,\xi(f)):=\Lambda(s+1,\chi)\,\int_{SL_2(\BZ)\backslash \mathbb{H}} \sum_{\mu\in D(L)}
\overline{\xi(f)_\mu (\tau)}\,E(\tau,s,\varphi_\mu)\,y^d\,\frac{dx\,dy}{y^2}.
\end{equation}
Here $ E(\tau,s,\varphi_\mu)$ is the restriction of $E(\vec{\tau},s,\varphi_\mu)$ to the diagonal embedding
$\BH \hookrightarrow \BH^d$. Then $\mathscr{L}(s,\xi(f))$ is entire, satisfies the functional equation and $\mathscr{L}(0,\xi(f))=0.$

\vskip 3mm
[AGHM] \cite[Theorem 6.4.2]{AGHM2} proved the following\,:
\begin{theorem}
Suppose $f\in H(1-\frac{n}{2},\rho_L)$ is a harmonic weak Maass form with {\sf integral} principal part
of weight $1-\frac{n}{2}$ with respect to $Mp(2,\BZ)$ and $\rho_L$. Let $\widehat{\mathscr Z}(f)$ be the arithmetic divisor on $\mathcal M$ (cf.\,(6.26)). Then one obtains the equality
\begin{equation}
\Lambda (0,\chi)\, [\widehat{\mathscr Z}(f):\mathcal Y ]= \deg_\BC (Y)
\left\{ a(0,0)\cdot c_f^+ (0,0)- \mathscr{L}' (0,\xi (f)) \right\}
\end{equation}
holds up to a $\BQ$-linear combination of $\{ \log (p)\,\vert\ p | D_{{\rm bad},L} \}$. Here
$D_{{\rm bad},L}$ is the product of certain bad primes, depending on the lattice $L\subset V\subset \mathscr{V}$,
\begin{equation*}
\deg_\BC (Y):=\sum_{y\in Y(\BC)} \frac{1}{ \vert {\rm Aut}(y) \vert }
\end{equation*}
is the number of $\BC$-points of the $E$-stack $Y$ counted with multiplicites,
$\Lambda (s,\chi)$ is the completed $L$-function, the constant $a(0,0)$ is the derivative at $s=0$ of the constant term of $E(\vec{\tau},s)$ and $c^+_f (0,0)$ denotes the $(0,0)$-th coefficient of the holomorphic part $f^+$ of $f$ (cf.\,(6.5)).
\end{theorem}
Also [AGHM] showed the identity
\begin{equation}
\frac{a(0,0)}{\Lambda(0,\chi)}=-2\cdot \frac{\Lambda'(0,\chi)}{\Lambda(0,\chi)}
\end{equation}
up to a $\BQ$-linear combination of $\{ \log (p)\,\vert\ p | D_{{\rm bad},L} \}$
(cf.\,\cite[Proposition 7.8.2]{AGHM2}). The proof of Theorem 7.2 is based on the
Bruinier-Kudla-Yang calculation \cite{BKY} of the values of the Green function $\Phi (f)$ at the points of
$\mathcal Y$ and the computation of the finite intersection multiplicities to the arithmetic intersection
$[\widehat{\mathscr Z}(f):\mathcal Y ]$ using the Breuil-Kisin theory.

\vskip 3mm
Assume $f\in M_{1-\frac{n}{2}}^!$ is weakly holomorphic so that $\xi (f)=0$ by the exact sequence (6.10).
Then $f$ has the Fourier expansion
\begin{equation}
f (\tau):=\sum_{\mu\in D(L)}  \sum_{ m\in\BQ,\,  m \gg -\infty }
c (m,\mu) \,q^n \,\phi_\mu, \quad \tau\in \mathbb H.
\end{equation}
It follows from Theorem 7.2 and the formula (7.10) that
\begin{equation}
\frac{[\widehat{\mathscr Z}(f):\mathcal Y ]}{\deg_\BC (Y)}\, \approx_L \, 2\,c_f^+ (0,0)\cdot
\frac{\Lambda'(0,\chi)}{\Lambda(0,\chi)},
\end{equation}
where $\approx_L$ means the equality up to a $\BQ$-linear combination of
$\{ \log (p)\,\vert\ p | D_{{\rm bad},L} \}$. The integral model $\mathcal M$ carries over it the tautological bundle $\omega$. Any $h\in H(\mathbb A_f)$ determines a uniformization
\begin{equation*}
{\bf D} \lrt M(\BC),\qquad z\longmapsto [(z,h)]
\end{equation*}
of a connected component of the complex fibre of $\mathcal M$, and provides the tautological bundle
on ${\bf D}$ pulling back $\omega$. If we endow $\omega$ with the metric $\| z \|^2=-(z,\overline{z})_Q$,
we obtain
\begin{equation*}
\widehat{\omega}=(\omega,\|\,\, ,\,\,\| )\in \widehat{\rm Pic}(\mathcal M).
\end{equation*}
After possibly replacing $f$ by a positive integer multiple, the theory of Borcherds product gives us
a rational section $\Psi (f)$ of the line bundle $\omega^{\otimes c_f (0,0)}$ such that
\begin{equation*}
-\log \| \Psi(f) \|^2 =\Phi (f)-c_f (0,0)\,\log (4\pi e^{\gamma}),\quad (\gamma\ is \ the \ Euler\ constant)
\end{equation*}
and satisfying the identity
\begin{equation*}
{\rm div} (\Psi (f))= \mathscr{Z} (f)
\end{equation*}
up to a linear combination of irreducible components of the special fibre $\mathcal{M}_{\mathbb F_2}.$

\vskip 2mm
We define a Cartier divisor
\begin{equation*}
\mathcal E_2 (f):={\rm div} (\Psi (f))- \mathscr{Z} (f)
\end{equation*}
on $\mathcal M$ supported entirely in characteristic $2$. If we endow $\mathcal E_2 (f)$ with the trivial Green function, we obtain a metrized line bundle
$\widehat{\mathcal E}_2 (f)\in \widehat{\rm Pic} (\mathcal M)$
satisfying the equality
\begin{equation*}
[ \widehat{\omega}^{\otimes c_f (0,0)}:\mathcal{Y}]=[\widehat{\mathscr Z}(f):\mathcal Y ]
-c_f (0,0)\,\log (4\pi e^{\gamma})\cdot d \,\deg_\BC (Y) + [\widehat{\mathcal E}_2 (f):\mathcal Y ].
\end{equation*}
We choose $f$ such that $c_f (0,0)\neq 0,$ then according to (7.9), we obtain
\begin{equation}
\frac{[\widehat{\omega}:\mathcal Y ]}{\deg_\BC (Y)}+d\cdot \log (4\pi e^{\gamma}) \, \approx_L \,
\frac{1}{c_f (0,0)}\cdot \frac{[\widehat{\mathcal E}_2 (f):\mathcal Y ]}{\deg_\BC (Y)}-
2\cdot \frac{\Lambda'(0,\chi)}{\Lambda(0,\chi)}.
\end{equation}
The cycle $\mathcal Y$ carries a canonical line bundle $\widehat{\omega}_0.$ [AGHM] showed (7.14) and (7.15)\,:
\begin{equation}
  \frac{1}{2^{d-2}} \sum_{\Phi} h_{\rm Fal}(E,\Phi)=
  \frac{\widehat{\deg}_{\mathcal Y} (\widehat{\omega}_0)}{\deg_\BC (Y)} + \log | d_F | -2\,d\cdot \log (2\pi)
\end{equation}
and
\begin{equation}
\frac{[\widehat{\omega}:\mathcal Y ]}{\deg_\BC (Y)} \,\approx_L \,
\frac{\widehat{\deg}_{\mathcal Y} (\widehat{\omega}_0)}{\deg_\BC (Y)} + \log | d_F |,
\end{equation}
where the sum on the left hand side in (7.14) runs over all $2^d$ CM types of $E$. Putting all these together, one finds that
\begin{equation}
  \frac{1}{2^d} \sum_{\Phi} h_{\rm Fal}(E,\Phi)=-\frac{1}{2}\cdot\frac{L'(0,\chi)}{L(0,\chi)}-\frac{1}{4}
  \log \Big| \frac{d_E}{d_F}\Big| - \frac{d}{2}\cdot\log (2\pi) +\sum_p b_E (p)\log (p)
\end{equation}
for some rational numbers $b_E (p)$ with $b_E (p)=0$ for all primes $p \nmid 2\,D_{{\rm bad},L}$. In fact,
$b_E (p)=0$ for all primes $p$. Finally one obtains the averaged Clomez's formula (7.4).

\begin{remark}
The main ideas of the proof of [YZ] \cite{YZ} are provided by two steps. The first step is to describe the averaged Faltings height in terms of the height of a single CM point on a quaternionic Shimura curve, and the second step is to prove that the height of a single CM point is given by the expected logarithmic derivative. Their proof is essentially based on the generalized Gross-Zagier formula proved by Yuan-Zhang-Zhang \cite{YZZ}.
\end{remark}

\end{section}

\vskip 10mm
\begin{section}{{\bf Final remarks}}
\setcounter{equation}{0}
\vskip 3mm
In the final section we give some remarks and open problems.
\vskip 2mm
We assume that the Birch-Swinnerton-Dyer conjecture (3.29) in section 3 holds. We take $K=\BQ$. Let $r:=r_E$ be the rank of an elliptic curve $E$ with $r\geq 2$. We assume that the Shafarevich-Tate group ${\rm III}(E)$ is finite. Then
\begin{equation}
L(E,s)={ {M\,|{\rm III}(E)|\,R_E}\over {|E(\BQ)_{\rm tor}|^2} }\,(s-1)^r\,+\,c_{r+1}(s-1)^{r+1}+\cdots ,
\end{equation}
where $R_E$ is the elliptic regulator of $E$ and $M=\prod_{v\in S_E} m_v$ is an explicitly written product of local Tamagawa factors over the set $S_E$ of all archimedean places of $\BQ$ and places where $E$ has bad reduction and $m_v=\int_{E(\BQ_v)}\omega,\ \omega$ being the N{\'e}ron differential of $E$.
We denote by $L^{(\ell)}(E,s)$ the $\ell$-th derivative of $L(E,s)$ with respect to $s$. It follows from (8.1) that
\begin{equation}
L^{(r)}(E,1)=r!\,{ {M\,|{\rm III}(E)|\,R_E}\over {|E(\BQ)_{\rm tor}|^2}} \neq 0.\end{equation}

\vskip 3mm \noindent
{\bf Problem 1.} Let $r,\, E,\, c_{r+d}\,(d=1,2.\cdots)$ with $r\geq 2$ as above. Compute the coefficients
$c_{r+1},\,c_{r+2},\cdots.$

\vskip 3mm
The Gross-Zagier formula relates the N{\'e}ron-tate heights of Heegner points on the modular curves to the central values of the first derivatives of certain Rankin-Selberg $L$-functions. Some generalizations of the Gross-Zagier formula was made by Zhang \cite{Zh1, Zh2}. Thereafter Yuan-Zhang-Zhang \cite{YZZ} proved a complete Gross-Zagier formula on quaternionic Shimura curves on totally real fields. We propose the following problem.

\vskip 3mm \noindent
{\bf Problem 2.} Let $\ell$ be a positive integer with $\ell \geq 2$. Discover an analogue of the
Gross-Zagier formula relating the central values of the $\ell$-th derivatives of certain (Rankin-Selberg) $L$-functions to geometric-arithmetic quantities.

\vskip 3mm
In section 6 we learned the relations between the Faltings height pairing of arithmetic special divisors and CM cycles on Shimura varieties associated to orthogonal groups of signature $(n,2)$, and central values of the first derivatives of the Rankin $L$-functions associated with the above Shimura varieties. In section 7 we learned the relations between the arithmetic intersection multiplicities of special divisors and big CM points on the Shimura varieties and the central values of the first derivatives of certain $L$-functions.
\vskip 2mm
We propose the following natural problem.
\vskip 3mm \noindent
{\bf Problem 3.} Let $\ell$ be a positive integer with $\ell \geq 2$. Study the relations between central values of the $\ell$-th derivatives of certain $L$-functions and geometric-arithmetic properties of the associated Shimura varieties.

\end{section}


\newpage
\newtheorem*{theorem*}{Theorem}
\newtheorem*{remark*}{Remark}
\appendix
{\bf \large Appendix\ A. The $L$-function $L(F,U,s)$}
\renewcommand{\theequation}{A.\arabic{equation}}
\vskip 3mm
\par We follows the notations in section 6. This note is based on \cite{BY2}. Let $(V,Q)$ be a quadratic space over $\mathbb{Q}$ of signature $(n,2)$.
Let $H:=\text{GSpin}(V)$ and $G:=\text{SL}_2$. We write $G^{\prime}_\mathbb{A}$ for the 2-fold metaplectic cover of $G_\mathbb{A}$. Let $G^{\prime}_\mathbb{R}$, $K^{\prime}_\infty$ and $K^{\prime}$ be the inverse images in $G^{\prime}_\mathbb{A}$ of $G(\mathbb{R})$, $K_\infty=\text{SO}(2,\mathbb{R})$
and $K=\text{SL}(2,\hat{\mathbb{Z}})\subset G(\mathbb{A}_f)$ respectively. we write $G'_\mathbb{Q}$ for the image in $G^{\prime}_\mathbb{A}$ of $G(\mathbb{Q})$ under the canonical splitting. Then we have
\begin{equation} \label{appeq01}
G_{\mathbb{A}}^{\prime}=G_{\mathbb{Q}}^{\prime} G_{\mathbb{R}}^{\prime} K^{\prime}\quad \text{ and }\quad \Gamma:=SL(2,\mathbb{Z})\cong G_{\mathbb{Q}}^{\prime} \cap G_{\mathbb{R}}^{\prime}K^{\prime}.
\end{equation}
We put $\Gamma^\prime=\text{Mp}(2,\mathbb{Z})$, i.e., the inverse image of $\Gamma$ in $G_{\mathbb{R}}^{\prime}$. Then for every $\gamma^\prime\in \Gamma^\prime$, there are unique elements $\gamma_0\in\Gamma$ and $\gamma^{\prime\prime}\in K$ such that
\begin{equation} \label{appeq02}
\gamma_{0}=\gamma^{\prime} \gamma^{\prime \prime}.
\end{equation}
Let $\psi$ be the standard non-trivial additive character of $\mathbb{A}/\mathbb{Q}$. Then $G_{\mathbb{A}}^{\prime}$ and $H(\mathbb{A})$ act on the space $S(V(\mathbb{A}))$ of Schwartz-Bruhat functions on $V(\mathbb{A})$ via the Weil representation $\omega=\omega_\psi$. For $\varphi\in S(V(\mathbb{A}))$, one has the usual theta function
\begin{equation}\label{appeq03}
\vartheta(g, h ; \varphi):=\sum_{v \in V(\mathbb{Q})}(\omega(g, h) \varphi)(v),\quad g \in G_{\mathbb{A}}^{\prime}, h \in H(\mathbb{A}).
\end{equation}
It is left invariant under $G_{\mathbb{Q}}^{\prime}$ and trivially left invariant under $H(\mathbb{Q})$.
For $z\in {\bf D}$, we consider the positive definite quadratic form $(\;,\;)_z$ on $V(\mathbb{R})$ defined by
$$
(v, v)_{z} :=\left(v_{z^{\bot}}, v_{z^{\bot}}\right)-\left(v_{z}, v_{z}\right),\quad v \in V(\mathbb{R}).
$$
It is easy to see that the Gaussian $\varphi_{\infty}(v, z):=\exp \left(-\pi(v, v)_{z}\right)$ belong to $S(V(\mathbb{R}))$, and
$\varphi_{\infty}(hv, hz)=\varphi_{\infty}(v, z)$ for all $h\in H(\mathbb{R})$. Moreover it has weight $n/2-1$ under the action of $K_\infty^{\prime}\subset G_{\mathbb{R}}^{\prime}$. If $\varphi_f\in S(V(\mathbb{A}_f))$, one obtains a theta function on $G_{\mathbb{A}}^{\prime}\times H(\mathbb{A})$ by putting
\begin{equation}\label{appeq04}
\theta\left(g, h ; \varphi_{f}\right) :=\vartheta\left(g, h ; \varphi_{\infty}\left(\cdot, z_{0}\right) \otimes \varphi_{f}(\cdot)\right)
\end{equation}
where $z_0\in {\bf D}$ as a fixed base point. This theta function can be written as a theta function on $\mathbb{H}\times{\bf D}\times\mathbb{H}(\mathbb{A}_f)$ in the usual way.
For $z\in{\bf D}$ we choose $h_z\in H(\mathbb{R})$ such that $h_z\cdot z_0 = z$. Then
$w\left(h_{z}\right) \varphi_{\infty}\left(\cdot, z_{0}\right)=\varphi_{\infty}(\cdot, z)$. For $\tau=x+i y \in \mathbb{H}$, we put
\begin{equation*}
g_{\tau}:=\left( \begin{array}{ll}{1} & {x} \\ {0} & {1}\end{array}\right) \left( \begin{array}{ll}{y^{\frac{1}{2}}} & {0} \\ {0} & {y^{-\frac{1}{2}}}\end{array}\right) \left( \begin{array}{cc}{y^{\frac{1}{2}}} & {x y^{-\frac{1}{2}}} \\ {0} & {y^{-\frac{1}{2}}}\end{array}\right)
\end{equation*}
and $g_{\tau}^{\prime}=\left(g_{\tau}, 1\right) \in G_{\mathbb{R}}^{\prime}$. So $g_{\tau}^{\prime} \cdot i=\tau$. We define
\begin{align}\label{appeq05}
	\theta\left(\tau, z, h_{f} ; \varphi_{f}\right):&=y^{-\frac{n}{4}+\frac{1}{2}} \vartheta\left(g_{\tau}^{\prime},\left(h_{z}, h_{f}\right);\varphi_{\infty}\left(\cdot, z_{0}\right) \otimes \varphi_{f}(\cdot)\right). \notag\\
&=y^{-\frac{n}{4}+\frac{1}{2}} \sum_{v \in V(\mathbb{Q})} \omega\left(g_{\tau}^{\prime}\right)\left(\varphi_{\infty}(\cdot, z) \otimes \omega\left(h_{f}\right) \varphi_{f}\right)(v)	\notag\\
&=y \sum_{v \in V(\mathbb{Q})} e^{2 \pi i\left(Q(v_{z^\bot})\tau+Q\left(v_{z}\right) \overline{\tau}\right)}\otimes \varphi_{f}\left(h_{f}^{-1} v\right),
\end{align}
where $\tau \in \mathbb{H}, z \in {\bf D} \text { and } h_{f} \in H\left(\mathbb{A}_{f}\right)$.
If $\gamma^{\prime}=(\gamma, \phi) \in \Gamma^{\prime}$ and $\gamma_{0}=\gamma^{\prime} \gamma^{\prime \prime}$ as in (\ref{appeq02}), we have, according to \cite[Lemma 1.1]{Ku2},
\begin{equation}\label{appeq06}
\theta\left(\gamma \cdot \tau, z, h_{f} ; \varphi_{f}\right)= \phi(\tau)^{n-2} \,\theta\left(\tau, z, h_{f} ; w_{f}\left(\gamma^{\prime\prime}\right)^{-1} \varphi_{f}\right)
\end{equation}
Let $L, L^{\prime}, S_{L}, \phi_{\mu},D(L) :=L^\prime / L, \varphi_{L}\cdots h_L$ as in section 6. We define a $S_L$-valued theta function by
\begin{equation}\label{appeq07}
\theta_{L}\left(\tau, z, h_{f}\right) :=\sum_{\mu\in D(L)} \theta\left(\tau, z, h_{f} ; \varphi_{\mu}\right)\varphi_{\mu}, \quad \tau \in \mathbb{H}, z \in {\bf D},\ h_{f} \in H\left(\mathbb{A}_{f}\right).
\end{equation}
Let $\gamma^{\prime} \in \Gamma$ with $\gamma_{0}=\gamma^{\prime} \gamma^{\prime \prime}$ as in (\ref{appeq02}). We put
 \begin{equation}\label{appeq08}
 \varphi_{L}\left(\gamma^{\prime}\right) :=\overline{\omega_{f}}\left(\gamma^{\prime \prime}\right).
 \end{equation}
Then $\varphi_{L}$ is the Weil representation of $\Gamma$ on $S_L$ defined by the formulas (6.1) and (6.2) in Section 6. According to (\ref{appeq06}) and (\ref{appeq07}), we obtain
\begin{equation}\label{appeq09}
\theta_{L}\left(\gamma \cdot \tau, z, h_{f}\right)=\phi(\tau)^{n-2} \varphi_{L}\left(\gamma^{\prime}\right)\theta_{L}\left(\tau, z, h_{f}\right).
\end{equation}
Assume $n$ is even. For $a\in \mathbb{G}_m$ and $b\in \mathbb{G}_a$, we put
$$
m(a) :=\left[a, 0,0, a^{-1}\right] \qquad\text { and } \qquad n(b)=[1,b,0,1].
$$
Let
$$
M :=\left\{ m(a) \,\middle|\, a \in \mathbb{G}_{m}\right\} \qquad
\text{ and } \qquad
N :=\left\{ n(b) \,\middle|\, b \in \mathbb{G}_{a}\right\}.
$$
Let $P=MN$ be the parabolic subgroup of $G$. Let $\chi_{v}$ denote the quadratic character of  $\mathbb{A}^{\times} / \mathbb{Q}^{\times}$ associated to $V$ defined by
$$
\chi_{v}(a)=\left(a,(-1)^{\dim V / 2} \operatorname{det}(V)\right)_{\mathbb{A}}
$$
where $\det(V)$ denotes the Gram determinant of $V$ and $(\cdot, \cdot)_{\mathbb{A}}$ is the Hilbert symbol on $\mathbb{A}$. For $s\in \mathbb{C}$, we let $I\left(s, \chi_{v}\right)$ be the principal series of $G({\mathbb{A}})$ induced by $\chi_{v}\cdot |\cdot|^s$. It consists of all smooth functions $\Phi(g, s)$ on $G(\mathbb{A})$ satisfying
$$
\Phi(n(b) m(a) g, s)=\chi_{v}(a) \left| a\right|^{s+1} \Phi(g, s)
$$
for all $b\in \mathbb{A}$ and $a\in \mathbb{A}^\times$. There is a $G({\mathbb{A}})$-intertwining operator
$$
\lambda : S(V(\mathbb{A})) \longrightarrow I\left(s_{0}, \chi_{v}\right), \quad \lambda(\varphi)(g) :=(\omega(g) \varphi)(0),
$$
where
$s_{0}=\frac{1}{2} \dim V-1$. A section $\Phi(s) \in I\left(s, \chi_{v}\right)$ is called {\sf standard} if its restriction to $K_{\infty} K$ is independent of $s$. Using the Iwasawa decomposition $G(\mathbb{A})=N(\mathbb{A}) M(\mathbb{A}) K_{\infty} K$, we see that the function $\lambda(\varphi) \in I\left(s_{0}, \chi_{v}\right)$ has a unique extension to a standard section $\lambda(\varphi,s) \in I\left(s, \chi_{v}\right)$ such that
$\lambda(\varphi,s_0)=\lambda(\varphi)$. \\
\vskip 3mm
For any standard section $\Phi(g,s)$, the Eisenstein series
\begin{equation}\label{appeq10}
E(g, s ; \Phi) :=\sum_{\gamma \in P(\mathbb{Q})\backslash G(\mathbb{Q})} \Phi(\gamma g, s)
\end{equation}
converges for Re$(s)>1$, and defines an automorphic form on $G(\mathbb{A})$. It has a meromorphic continuation to the whole complex plane and satisfies a functional equation relating $E(g, s ; \Phi)$ and $E(g, -s ; M(s)\Phi)$.
\vskip 3mm
{\bf Theorem A.1}\,[The Siegel-Weil formula] {\it Let $V$ be a rational quadratic space of siguature $(n,2)$. Assume $V$ is anisotropic or that $\operatorname{dim} V-r_{0}>2$, where $r_0$ is the Witt index of $V$. Then $E(g, s ; \lambda(\varphi))$ is holomorphic at $s_0$ and
\begin{equation}\label{appeq11}
\frac{\alpha}{2} \int_{SO(V)(\mathbb{R})\backslash SO(V)(\mathbb{A})} \vartheta(g, h ; \varphi) d h = E\left(g, s_{0} ; \lambda(\varphi)\right),
\end{equation}
where $dh$ is the Tamagawa measure on $SO(V)(\mathbb{A})$, and $\alpha=2$ if $n=0$, and $\alpha=1$ if $n>0$.}
\begin{proof}
	The proof can be found in \cite[Theorem 4.1]{Ku2}.
\end{proof}
For the present being, we consider the special case where $(V,Q)$ is a definite space over $\mathbb{Q}$ of signature $(0,2)$. Then $(V,Q)$ is isomorphic to $(k,-cN(\cdot))$ for an imaginary quadratic field $k$ with the negative of the norm scaled by a constant $c\in \mathbb{Q}^+$. Then $H(\mathbb{Q}) \cong k^{\times}$ and $SO(V)$ is the group of norm $1$ elements in $k$. The homomorphism $H\rightarrow SO(V)$ is given by $h\longmapsto h\overline{h}^{-1}$, and $SO(V)$ acts on $k$ by multiplication. The Grassmannian $\bf D$ consists of $z_{v}^{+} \text { and } z_{v}^{-}$ given by $V(\mathbb{R})$ with positive and negative orientation respectively. For $\ell\in \mathbb{Z}$, we define a $S_L$-valued Eisenstein series of weight $\ell$ by
\begin{equation}\label{appeq12}
E_{L}(\tau, s ; \ell):=y^{-\frac{\ell}{2}} \sum_{\mu \in D(L)}E\left(g_{\tau}, s ;\Phi_{\infty}^{\ell} \otimes \lambda_{f}\left(\phi_{\mu}\right)\right)\phi_{\mu}.
\end{equation}
We normalize the measure on $SO(V)(\mathbb{R})$(resp.\,$SO(V)(\mathbb{A}_f)$) so that vol$(SO(V)(\mathbb{R}))=1$ (resp. vol$(SO(V)(\mathbb{Q})\backslash SO(V)(\mathbb{A}_f))=2$). Then according to Theorem A.1, we obtain the Siegel-Weil formula
\begin{equation}\label{appeq13}
\int_{SO(V)(\mathbb{Q})\backslash SO(V)(\mathbb{A}_f)} \vartheta_L(\tau, z^\pm ; h_f) d h_f=E_{L}(\tau, 0 ;-1).
\end{equation}
If $L_k$\,(resp.\,$R_k$) is the Maass lowering\,(resp.\,raising) operator (cf.\,(6.7) and (6.8) in section 6), we see easily that
$$
L_{\ell} E_{L}(\tau, s ; \ell)=\frac{1}{2}\,(s+1-\ell)\, E_{L}(\tau, s ; \ell-2)
$$
and
$$
R_{\ell} E_{L}(\tau, s ; \ell)=\frac{1}{2}\,(s+1+\ell)\, E_{L}(\tau, s ; \ell+2).
$$
Thus we have
\begin{equation}\label{appeq14}
L_{1} E_{L}(\tau, s ; 1)=\frac{s}{2}\,E_{L}(\tau, s ; -1).
\end{equation}
We see that $E_{L}(\tau, 0 ; 1)=0$ because $E_{L}(\tau, s ; -1)$ is holomorphic at $s=0$. And $E_{L}(\tau, s ; 1)$ is an incoherent Eisenstein series and hence satisfies an odd functional equation under $s \mapsto 1-s$. From (\ref{appeq14}), We see that

\begin{equation}\label{appeq15}
L_{1} E_{L}^\prime(\tau, 0 ; 1)=\frac{1}{2}\,E_{L}(\tau, 0 ; -1),
\end{equation}
where $E_{L}^\prime(\tau, s ; 1)$ denotes the derivative of $E_{L}(\tau, s ; 1)$ with respect to $s$. \\
As in \cite{BY2, Sch}, we write the Fourier expansion of $E_{L}(\tau, s ; 1)$ in the form
\begin{equation}\label{appeq16}
E_{L}(\tau, s ; 1)=\sum_{\mu \in D(L)} \sum_{m \in \mathbb{Q}} A_{\mu}(s, m, y)\, q^{m} \,\phi_\mu,
\end{equation}
where $\tau=x+i y \in \mathbb{H}, \text { and } q=e^{2 \pi i \tau}$.
We note that $A_{\mu}(s, m, y)=0$ unless $m \in Q(\mu)+\mathbb{Z}$. \\
\indent
Since $E_{L}(\tau, 0 ; 1)=0$, the coefficients have a Laurent expansion of the form
\begin{equation}\label{appeq17}
A_{\mu}(s, m, y)=b_{\mu}\left(m, y\right) s+O\left(s^{2}\right)
\end{equation}
at $s=0$, and we have
\begin{equation}\label{appeq18}
E_{L}^\prime(\tau, 0 ; 1)=\sum_{\mu \in D(L)} \sum_{m \in \mathbb{Q}} b_{\mu}\left(m, y\right) q^{m} \phi_\mu.
\end{equation}
We define
\begin{equation}\label{appeq19}
\kappa (m, \mu) :=\left\{\begin{array}{l}{\displaystyle\lim _{y \rightarrow \infty} b_{\mu}(m, y)\qquad\quad\,\text {,\;if } m \neq 0 \text { or } \mu \neq 0 \text {,}} \\ {\displaystyle\lim _{y \rightarrow \infty} b_{0}(0, y)-\log y \text { ,\;if } m=0 \text { and } \mu=0. }\end{array}\right
.\end{equation}
According to \cite[Proposition 2.19 and Lemma 2.20]{Sch}, the limits exist. If $m>0$, then $b_{\mu}(m, y)$ is independent of $y$ and equal to $\kappa(m,\mu)$. We also $\kappa(m,y)=0$ for $m<0$ or $m=0,\mu=0$. Now we define a holomorphic $S_L$-valued function on $\mathbb{H}$ by
\begin{equation}\label{appeq20}
\mathscr{E}_{L}(\tau) :=\sum_{\mu \in D(L)} \sum_{m \in \mathbb{Q}} \kappa(m, y)\,q^m \,\phi_{\mu}.
\end{equation}
\begin{remark*}
$E_{L}^\prime(\tau, 0 ; 1)$ is a harmonic weak Maass form of weight $1$ and
$$
\xi_1(E_{L}^\prime(\tau, 0 ; 1))=y^{-1}E_{L}(\tau, 0 ; 1).
$$
We note that $\mathscr{E}_{L}(\tau)$ is the holomorphic part of $E_{L}^\prime(\tau, 0 ; 1)$.
\end{remark*}
Now we consider the lattice $(L,Q)=\left(\mathfrak{a},-\frac{N}{N(\mathfrak{a})}\right)$, where $\mathfrak{a}$ is a fractional ideal of an imaginary quadratic field $k=\mathbb{Q}(\sqrt{D})$ with fundamental discriminent $D\equiv 1(\text{mod } 4)$. Let $\mathcal{O}_k$ be the ring of integers in $k$, and let $\partial$ be the different of $k$. In this case, $V=k, L^\prime=\partial^{-1}\mathfrak{a}$ and
$$
D(L) :=L^{\prime} / L=\partial^{-1} \mathfrak{a} /\mathfrak{a} \cong \partial^{-1} / \mathcal{O}_{k} \cong \mathbb{Z} / D \mathbb{Z}.
$$
Let $\chi_D$ be the quadratic Dirichlet character associated to $k/\mathbb{Q}$. We set
\begin{equation}\label{appeq21}
\Lambda\left(s, \chi_{D}\right)=\left|D\right|^{\frac{s}{2}} \Gamma_{\mathbb{R}}{(s+1)} L\left(s, \chi_{D}\right), \quad\Gamma_{\mathbb{R}}{(s)}:=\pi^{-\frac{s}{2}}\,\Gamma\left(\frac{s}{2}\right).
\end{equation}
and
\begin{equation}\label{appeq22}
E_{L}^{*}(\tau, s):=\Lambda(s+1,\chi_D)\, E_L(\tau, s ; 1).
\end{equation}
Bruimier and Yang \cite{BY2} proved that $E_{L}^{*}(\tau,-s)=-E_{L}^{*}(\tau, s).$
We observe that $\Lambda(s,\chi_D)=\Lambda(1-s,\chi_D)$ and
\begin{equation}\label{appeq23}
\Lambda(1,\chi_D)=\frac{\sqrt{D}}{\pi} L(1,\chi_{D})= 2 \cdot \frac{h_k}{\omega_{k}},
\end{equation}
where $h_k$ is the class number of $k$ and $\omega_{k}$ is the number of roots of unity in $k$. \\
\par Finally we go back to the quadratic space $(V,Q)$ over $\mathbb{Q}$ of signature $(n,2)$. Let $L$ be an even lattice and $U\subset V$ a negative definite two dimentional $\mathbb{Q}$-space of $V$. We put
$$
N=L\cap U \qquad \text{ and } \qquad S=L\cap V_+ \qquad\text{(cf.\,(6.19))}.
$$
For any $F\in S_{1+\frac{n}{2},\rho_L}$, we define an $L$-function $L(F,U,s)$ by means of the convolution integral
\begin{eqnarray}\label{aapeq24}
L(F,U,s):&=&\left(\theta_{S}(\tau) \otimes E_{N}(\tau, s ; 1),F(\tau)\right)_\text{Pet} \\
&=& \int_{\mathfrak{F}}\langle\theta_{S}(\tau) \otimes E_{N}(\tau, s ; 1),\overline{F(\tau)}\rangle \, y^{1+\frac{n}{2}}\, d\mu(\tau),\nonumber
\end{eqnarray}
where $\mathfrak{F}=\left\{\tau \in \mathbb{H} \,\middle| \,|\text{Re}(\tau)|\leq \frac{1}{2},\,|\tau|\geq 1 \right\}$ and $d\mu(\tau)=y^{-2}dxdy$.


\vskip 12mm

\appendix
\renewcommand{\theequation}{A.\arabic{equation}}

{\bf \large Appendix\ B. The Andr{\'e}-Oort Conjecture}
\vskip 5mm
In this section we review recent progress on the Andr{\'e}-Oort conjecture quite briefly.
\vskip 3mm\noindent
{\bf Definition\ B.1.}
{\it Let $(G,X)$ be a Shimura datum and let $K$ be a compact open subgroup of $G({\mathbb A}_f).$ We let
\begin{equation*}
Sh_K(G,X):=G(\BQ)\ba X \times G({\mathbb A}_f)/K
\end{equation*}
be the Shimura variety associated to $(G,X).$ An algebraic subvariety $Z$ of the Shimura variety $Sh_K(G,X)$ is said to be
{\sf weakly special} if there exist a Shimura sub-datum $(H, X_H)$ of $(G,X)$, and a decomposition
\begin{equation*}
(H^{\rm ad},X_H^{\rm ad})=(H_1,X_1) \times (H_2,X_2)
\end{equation*}
and $y_2\in X_2$ such that $Z$ is the image of $X_1\times \{ y_2\}$ in $Sh_K(G,X)$. Here $(H^{\rm ad},X_H^{\rm ad})$ denotes
the adjoint Shimura datum associated to $(G,X)$ and $(H_i,X_i)\ (i=1,2)$ are Shimura data. In this definition,
a weakly special subvariety is said to be ${\sf special}$ if it contains a special point and $y_2$ is special.}
\vskip 0.25cm
Andr{\'e} \cite{A} and Oort \cite{O} made conjectures analogous to the Manin-Mumford conjecture where the ambient variety is a Shimura
variety (the latter partially motivated by a conjecture of Coleman \cite{Co}). A combination of these has become known as the
Andr{\'e}-Oort conjecture (briefly the A-O conjecture).
The Andr{\'e}-Oort
\vskip 0.35cm\noindent
{\bf A-O Conjecture.} {\it Let $S$ be a Shimura variety and let $\Sigma$ be a set of special points in $S$. Then every irreducible component of the Zariski closure of $\Sigma$ is a special subvariety.}

\vskip 3mm\noindent
{\bf Definition\ B.2.} \cite{Pi1, Pi2}
{\it A {\sf pre-structure} is a sequence $\Sigma =(\Sigma_n:\,n\geq 1)$ where each $\Sigma_n$ is a
collection of subsets of $\BR^n$. A pre-structure $\Sigma$ is called a {\sf structure} over
the real field if, for all $n,m\geq 1 $ with $m\leq n$, the following conditions are satisfied:
\\
\indent (1) $\Sigma_n$ is a Boolean algebra (under the usual set-theoretic operations);\\
\indent (2) $\Sigma_n$ contains every semi-algebraic subset of $\BR^n$;\\
\indent (3) if $A\in \Sigma_m$ and $B\in\Sigma_n$, then $A\times B\in \Sigma_{m+n}$;\\
\indent (4) if $n\geq m$ and $A\in \Sigma_n$, then $\pi_{n,m}(A)\in \Sigma_m$, where
$\pi_{n,m}:\BR^n\lrt \BR^m$ is a coordinate\\
\indent\ \ \ \ \ projection on the first $m$ coordinates.
\vskip 0.3cm
\noindent If $\Sigma$ is a structure, and, in addition,\\
\vskip 0.05cm
\indent (5) the boundary of every set in $\Sigma_1$ is finite,\\
\vskip 0.05cm
\noindent
then $\Sigma$ is called an {\sf o-minimal structure} over the real field.
\vskip 0.2cm
If $\Sigma$ is a structure and $Z\subset \BR^n$, then we say that $Z$ is
{\sf definable} in $\Sigma$ if $Z\in \Sigma_n$. A function $f:A\lrt B$ is {\sf definable }
in a structure $\Sigma$ if its graph is definable, in which case the domain $A$ of $f$ and
image $f(A)$ are also definable by the definition.
If $A,\cdots,f,\cdots$ are sets or functions, then we denote by $\BR_{A,\cdots,f,\cdots}$
the smallest structure containing $A,\cdots,f,\cdots$. By a {\sf definable family of sets}
we mean a definable subset $Z\subset \BR^n\times\BR^m$ which we view as a family of fibres
$Z_y\subset \BR^n$ as $y$ varies over the projection of $Z$ onto $\BR^m$ which is definable,
along with all the fibres $Z_y$. A family of functions is said to be definable if the family
of their graphs is. A {\sf definable set} usually means a definable set in some o-minimal
structure over the real field.}

\vskip 3mm\noindent
{\bf Remark\ B.1.}
{\it The notion of a o-minimal structure grew out of work van den Dries \cite{DR1, DR2} on Tarski's problem concerning the decidability of the real ordered field with the exponential function, and
was studied in the more general context of linearly ordered structures
by Pillay and Steinhorn\,\cite{PS}, to whom the term ``o-minimal" (``order-minimal") is due.}

\vskip 0.35cm
In 2011 Pila gave a unconditional proof of the A-O conjecture for arbitrary products of modular curves using the theory of o-minimality.

\vskip 3mm\noindent
{\bf Theorem\ B.1.}
{\it Let
\begin{equation*}
X=Y_1\times \cdots \times Y_n\times E_1\times\cdots \times E_m \times {\mathbb G}_m^{\ell},
\end{equation*}
where $n,m,\ell\geq 0,\ Y_i=\G_{(i)}\ba \BH_1 (1\leq i\leq n)$ are modular curves corresponding to congruence subgroups $\G_{(i)}$ of
$SL(2,\BZ)$ and $E_j\,(1\leq j\leq m)$ are elliptic curves defined over ${\overline \BQ}$ and ${\mathbb G}_m$ is the multiplicative group.
Suppose $V$ is a subset of $X$. Then $V$ contains only a finite number of maximal special subvarieties.}
\vskip 0.23cm\noindent
{\it Proof.} See Pila \cite{Pi1}, Theorem 1.1.
\hfill $\Box$

\vskip 0.25cm
In 2013 Peterzil and Starchenko proved the following theorem using the theory of o-minimality.

\vskip 3mm\noindent
{\bf Theorem\ B.2.}
{\it The restriction of the uniformizing map $\pi:\BH_g\lrt {\mathcal A}_g$ to the classical fundamental domain for the Siegel modular group
$Sp(2g,\BZ)$ is definable.}
\vskip 0.23cm\noindent
{\it Proof.} See Peterzil and Starchenko \cite{P-S1,P-S2}.
\hfill $\Box$

\vskip 0.25cm
In 2014 Pila and Tsimerman gave a conditional proof of the A-O conjecture for the Siegel modular variety $\mathcal A_g.$

\vskip 3mm\noindent
{\bf Theorem\ B.3.}
{\it If $g\leq 6,$ then the A-O conjecture holds for $\mathcal A_g.$ If $g\geq 7$, the A-O conjecture holds for $\mathcal A_g$ under the
assumption of the Generalized Riemann Hypothesis (GRH) for CM fields.}
\vskip 0.23cm\noindent
{\it Proof.} See Pila-Tsimerman \cite{PT1,PT2}.
\hfill $\Box$

\vskip 0.25cm
Quite recently using Galois-theoretic techniques and geometric properties of Hecke correpondences, Klingler and Yafaev proved
the A-O conjecture for a general Shimura variety, and independently using Galois-theoretic and ergodic techniques Ullmo and Yafaev
proved the A-O conjecture for a general Shimura variety, under the assumption of the GRH for CM fields or another suitable assumption.
The explicit statement is given as follows.

\vskip 3mm\noindent
{\bf Theorem\ B.4.}
{\it Let $(G,X)$ be a Shimura datum and $K$ a compact open subgroup of $G({\mathbb A}_f)$. Let $\Sigma$ be a set of special points in $Sh_K(G,X)$. We make one of the two following assumptions\,:
\vskip 0.25cm
(1) Assume the GRH for CM fields.
\vskip 0.15cm
(2) Assume that there exists a faithful representation $G\hookrightarrow GL_n$ such that with respect to this representation, the
Mumford-Tate group $MT(s)$ lie in one $GL_n(\BQ)$-conjugacy class as $s$ ranges through $\Sigma.$ Then every irreducible component of $\Sigma$ in $Sh_K(G,X)$ is a special subvariety.}
\vskip 0.23cm\noindent
{\it Proof.} See Klingler-Yafaev \cite{K-Y} and Ullmo-Yafaev \cite{U-Y}.
\hfill $\Box$

\vskip 3mm\noindent
{\bf Remark\ B.2.}
{\it We refer to \cite{Pi2} for the theory of o-minimality and the A-O conjecture. We also refer to \cite{G} for the A-O conjecture for mixed Shimura varieties.}

\vskip 3mm
Now we give a sketch of Tsimerman's proof of A-O conjecture for the Siegel modular variety
$ \mathcal A_g \ (g\geq 1$).
\vskip 2mm
Let $E$ be a CM field with totally real field $F$. We set $g=[F:\BQ]$ and let $\Phi$ be a CM type of $E$. Define $S(E,\Phi)$ to be the set of all isomorphism classes of complex abelian varieties of dimension $g$ together with an embedding $\mathcal O_E \hookrightarrow {\rm End}_\BC (A)$ such that the induced action of $E$ on the tangent space $T_0 (A(\BC))\cong \BC^g$ is given by $\Phi$.

\vskip 3mm\noindent
{\bf Proposition\ B.1.} {\it For a primitive CM type $\Phi$, there exist $A_1,\,A_2\in S(E,\Phi)$ such that
\begin{equation*}
\deg^{\ell} (A_1,A_2)\geq |d_E|^{1/4-o_g(1)},
\end{equation*}
where $\deg^{\ell} (A_1,A_2)$ denotes the lowest degree among the degrees of all isogenies between $A_1$ and $A_2$.}
\begin{proof} The proof can be found in \cite[Proposition 2.2]{T}.\end{proof}

\vskip 3mm\noindent
{\bf Masser-W{\"u}stholz\ Isogeny Theorem\,:} {\it Let $A$ and $B$ be two abelian varieties of dimension $g$
over a number field $k$, and suppose that there exists an isogeny between them over $\BC$. Let $N$ be the minimal degree of isogenies between them over $\BC$. Then we have the following bound\,:
\begin{equation*}
N \ll_g \left( {\rm max}(h_{\rm Fal}(A),[k:\BQ])\right)^{c_g},
\end{equation*}
where $c_g$ is a positive constant depending only on $g$.}
\vskip 2mm
We refer to \cite{M-W} for more detail.

\vskip 3mm
In this setting, the averaged Colmez's formula can be written as follows\,:
\begin{equation*}
\noindent
{\bf (ACF)}\hskip 10mm \ \ \sum_{\Phi} h_{\rm Fal}(E,\Phi)=\sum_{\Phi}\left( \sum_{\rho} c_{\rho,\Phi}
\left( \frac{L'(0,\rho)}{L(0,\rho)}+ \frac{\log f_\rho}{2} \right)\right), \hskip 35mm
\end{equation*}
where $\Phi$ runs over all $2^g$ CM types of $E$, $\rho$ runs over irreducible complex representations of the Galois group of the normal closure $E^{\rm nc}$ of $E$ for which $L(0,\rho)$ does not vanish, $c_{\rho,\Phi}$ are rational numbers depending only on the finite combinatorial data given by $\Phi$ and the Galois group of $E^{\rm nc}$, and $f_\rho$ is the Artin conductor of $\rho$.
\vskip 3mm
Using the above averaged Colmez's formula (ACF) and the following bounds (B1), (B2) and (B3)\,:
\par
(B1) \ \ $L(1,{\bar \rho})=| d_E |^{o_g (1)}\,;$
\par
(B2) \ \ $L'(1,{\bar \rho})\leq | d_E |^{o_g (1)}\,;$
\par
(B3) \ \ $h_{\rm Fal} (B) \geq O_g (1)$ for any abelian variety $B$ of dimension $g$,
\vskip 3mm\noindent
Tsimerman \cite[Corollary 3.3]{T} proved the following proposition\,:
\vskip 3mm\noindent
{\bf Proposition\ B.2.} {\it Let $E$ be a CM field with $[E:\BQ]=2g,\ \Phi$ be a primitive CM type and
$A\in S(E,\Phi)$. Then we have the bound
\begin{equation*}
h_{\rm Fal} (A) \leq | d_E |^{o_g (1)}.
\end{equation*}}

\vskip 2mm
The bound (B1) follows from the classical Brauer-Siegel theorem (cf.\,\cite{BR})  the bound (B2) from a standard sub-convexity estimate for $L(s,{\bar \rho})$ (cf.\,\cite{I-K}), and the statement (B3) is due to the result of Bost (cf.\,\cite{Boj}).

\vskip 3mm
Combining Proposition B.1, Masser-W{\"u}stholz Isogeny Theorem and Proposition B.2, Tsimerman \cite[Theorem 4.2]{T} proved the following theorem.
\vskip 3mm\noindent
{\bf Theorem\ B.5.} {\it There exists a constant $\delta_g$ such that if $E$ is a CM field of degree $2g,\ \Phi$ is a primitive CM type of $E$, and $A\in S(E,\Phi)$, the the field of moduli $\BQ (A)$ of $A$ satisfies
\begin{equation*}
[\BQ (A):\BQ] \gg | d_E |^{\delta_g}.
\end{equation*}}

\vskip 2mm
For $x\in \mathcal A_g (\overline{\BQ})$, let $A_x$ denote the corresponding ppav of dimension $g$,
let $R_x$ be the center of the endomorphism algebra ${\rm End}_{\overline{\BQ}}(A_x)$, and ${\rm Disc}(R_x)$ the discriminant of $R_x.$ Finally, using Theorem B.5 together with some properties, Tsimerman \cite[Theorem 5.2]{T} proved Edixhoven's conjecture \cite{EMO}.

\vskip 3mm\noindent
{\bf Theorem\ B.6\ (Edixhoven's\ conjecture).} {\it Let $g\geq 1.$ Then there exists a constant $b_g >1$ depending only on $g$ such that, for a special point $x\in \mathcal A_g (\overline{\BQ})$,
\begin{equation*}
  | {\rm Disc}(R_x)| \ll_g | {\rm Gal}(\overline{\BQ}/\BQ)\cdot x |^{b_g}
\end{equation*}
with the implied constants depending on $g$.}
\begin{proof} The detailed proof can be found in \cite[Theorem 5.2]{T}.\end{proof}

\vskip 3mm
The fact that Theorem B.6 implies the A-O conjecture for the Siegel modular varieties
$\mathcal A_g \ (g\geq 1)$ was proved by Pila and Tsimerman \cite[Theorem 7.1]{PT2}. This proof is based on the theory of o-minimality, point-counting and the works of Bombieri, Zannier, Wilkie et al.

\vskip 12mm

\appendix
{\bf \large Appendix\ C. The Gross-Kohnen-Zagier theorem in higher \\
\indent \ \ \ \ \ \indent \ \ \ \ \ \ \ \ \ \ \ dimension}
\renewcommand{\theequation}{C.\arabic{equation}}
\def\k{\kappa}
\vskip 3mm
We let $\Gamma$ be a subgroup of $Mp_2(\BR)$ commensurable with $Mp_2(\BZ)$, and let $\rho$ be a finite dimensional representation of $\G$
on a complex vector space $V_\rho$ which factors through a finite quotient of $\G$ such that $\rho=\s_k$ on
$\G\cap K$. Choose $k\in {\frac 12}\BZ$. We denote by ${\sf ModForm}(\G,k,\rho)$ the space of modular forms of weight $k$ and $\rho$ for
$\G$ which are meromorphic at cusps and by ${\sf HolModForm}(\G,k,\rho)$ the space of modular forms of weight $k$ and $\rho$ for
$\G$ which are holomorphic at all cusps.
\vskip 0.2cm
Let $\kappa$ be a cusp of $\G$ and let $q_{\kappa}$ be a uniformizing parameter at $\kappa$ in $\G\ba \BH$. let $(\rho^*,V_\rho^*)$ be the
dual representation of $(\rho,V_\rho).$ Let
\begin{equation*}
{\sf PowSer}_\kappa(\G,\rho)=\BC [[q_\kappa]]\otimes V_\rho \qquad {\rm and}\qquad
{\sf Laur}_\kappa (\G,\rho)=\BC [[q_\kappa]][q_\kappa^{-1}]\otimes V_\rho
\end{equation*}
be the space of formal power series in $q_\kappa$ with coefficients in $V_\rho$ and the space of formal Laurent series in $q_\kappa$
with coefficients in $V_\rho$ respectively. Let
\begin{equation*}
{\sf Sing}_\kappa(\G,\rho)=
{\frac {{\sf Laur}_\kappa (\G,\rho)}{q_\k {\sf PowSer}_\kappa(\G,\rho) } }
\end{equation*}
be the space of possible singularities and constant terms of $V_\rho$-valued Laurent series at $\k$.
The two spaces ${\sf PowSer}_\kappa(\G,\rho^*)$ and ${\sf Sing}_\kappa(\G,\rho)$ are paired into $\BC$ by taking the residue
\begin{equation*}
\langle f,\phi\rangle ={\rm res}\big( f\phi\, q_\k^{-1}dq_\k\big),\qquad f\in {\sf PowSer}_\kappa(\G,\rho^*),\ \
\phi\in {\sf Sing}_\kappa(\G,\rho),
\end{equation*}
where the product of $f$ and $\phi$ is defined using the pairing of $V_\rho$ and $V_\rho^*$. The space
\begin{equation*}
{\sf Sing}(\G,\rho)=\bigoplus_{\k:\,cusp} {\sf Sing}_\kappa(\G,\rho)
\end{equation*}
and
\begin{equation*}
{\sf PowSer}(\G,\rho^*)= \bigoplus_{\k:\,cusp} {\sf PowSer}_\kappa(\G,\rho^*)
\end{equation*}
are paired by the sum of the local pairings at the cusps. Then we have the canonical maps
\begin{equation*}
\lambda_k:{\sf HolModForm}(\G,k,\rho)\lrt {\sf PowSer}_\kappa(\G,\rho)
\end{equation*}
and
\begin{equation*}
\lambda_k^*:{\sf HolModForm}(\G,k,\rho^*)\lrt {\sf PowSer}_\kappa(\G,\rho^*).
\end{equation*}
We also have the natural map
\begin{equation*}
\mu_{2-k}:{\sf ModForm}(\G,2-k,\rho)\lrt {\sf Sing}(\G,\rho)
\end{equation*}
defined by taking the Fourier expansions of their nonpositive part at the various cusps.

\vskip 0.2cm
We define the space ${\sf Obstruct}(\G,k,\rho)$ of obstructions to finding a modular form of weight $k$ and type $\rho$ which is holomorphic
on $\BH$ and has meromorphic singularities and constant terms at the cusps to be the space
\begin{equation*}
{\sf Obstruct}(\G,k,\rho)= {\frac {{\sf Sing}_\kappa(\G,\rho)}{\mu_k\left( {\sf ModForm}(\G,k,\rho)\right)}}
\end{equation*}

Borcherds \cite{Bo2} proved the following duality that is called the {\sf Serre\ duality}.
\vskip 3mm\noindent
{\bf Theorem\ C.1.}
{\it Suppose that $k\in {\frac 12}\BZ,\ \G$ is a subgroup of $Mp_2(\BR)$ commensurable with $Mp_2(\BZ)$,
and let $\rho$ be a finite dimensional representation of $\G$
on a complex vector space $V_\rho$ which factors through a finite quotient of $\G$ such that $\rho=\s_k$ on
$\G\cap K$. Then
\vskip 0.2cm\noindent
(a) $\dim_\BC {\sf Obstruct}(\G,2-k,\rho)< \infty$.

\vskip 0.2cm\noindent
(b) ${\sf Obstruct}(\G,2-k,\rho)$ is dual to ${\sf HolModForm}(\G,k,\rho)$. The pairing between them is induced by the above pairing between
${\sf Sing}(\G,\rho)$ and  ${\sf PowSer}(\G,\rho^*)$. In other words,
\begin{equation*}
\mu_{2-k}\left( {\sf ModForm}(\G,2-k,\rho)\right)= \Big( \lambda_k^*\big( {\sf HolModForm}(\G,k,\rho^*)\big)\Big)^\perp
\end{equation*}
equivalently,
\begin{equation*}
\lambda_k^*\big( {\sf HolModForm}(\G,k,\rho^*)\big)=
\Big( \mu_{2-k}\left( {\sf ModForm}(\G,2-k,\rho)\right)\Big)^\perp.
\end{equation*}}

\vskip 0.2cm\noindent
{\it Proof.} The proof can be found in \cite[p.\,224]{Bo2}. \hfill $\square$

\vskip 0.2cm
Let $M$ be an even lattice of signature $(2,\ell)$. Suppose that $\G$ is a discrete group acting on $Gr(M)$.
We define a divisor on $X_\G:=\G\ba Gr(M)$ to be a locally finite $\G$-invariant divisor on $Gr(M)$
whose support is a locally finite union of a finite number of $\G$-orbits of irreducible subvarieties of $Gr(M)$ of codimension one.
For any negative norm $v\in M\otimes \BR$, we let
\begin{equation*}
v^\perp=\left\{ x\in Gr(M)\,|\ x\perp v\ \right\}
\end{equation*}
be the divisor of $Gr(M)$. For any rational number $n\in\BQ$ and $\g\in M'/M$, we define the
{\sf Heegner\ divisor} $y_{n,\g}$ of $Gr(M)$ by
$$
{\rm (C.1)}\hskip 3cm     y_{n,\g}:=\sum_{v\in M+\g \atop (v,v)=2n} v^\perp .\hskip 10cm
$$
Since $v\in M+\g$ implies $-v\in M-\g$, we see that $y_{n,\g}=y_{n,-\g}.$ We define the group
$${\rm (C.2)}\hskip 3cm
{\sf Heeg}(X_\G)=\BZ\, y_{0,0}\bigoplus \sum_{\g\in M'/M}\sum_{n\in \BQ} \BZ \,y_{n,\g},
\hskip 10cm
$$
where $y_{0,0}$ is a symbol and $y_{n,\g}$ is the Heegner divisor of $Gr(M)$ defined by (C.1) such that
$y_{n,\g}=0$ if $n>0$, and $y_{n,\g}=0$ if $n=0$ and $\g=0$.
We say that a Heegner divisor $D$ is {\it principal} if it is of the form $D=c_{0,0}\,y_{0,0}+D_0$,
where $D$ is the divisor of a meromorphic automorphic form of weight $c_{0,0}/2$
for some integer $c_{0,0}\in\BZ$ and some unitary character of finite order of the subgroup of
${\rm Aut}(M)$ fixing all the elements of $M'/M$. We denote by ${\sf PrinHeeg}(X_\G)$ the subgroup of principal Heegner divisors
and define the group ${\sf HeegCl}(X_\G)$ of Heegner divisor classes by
$${\rm (C.3)}\hskip 3cm
{\sf HeegCl}(X_\G):= {\frac {{\sf Heeg}(X_\G)}{{\sf PrinHeeg}(X_\G)} }. \hskip 10cm
$$
We define the linear map
$${\rm (C.4)}\hskip 3cm
\xi_M: {\sf Sing}(Mp_2(\BZ),\rho_M)\lrt {\sf HeegCl}(X_\G)\otimes_\BZ \BC \hskip 10cm
$$
by
\begin{equation*}
\xi_M \left( \sum_{n,\g}c_{n,\g}\,q^n\,{\mathfrak e}_\g\right) =\sum_{n,\g} c_{n,\g}\,y_{n,\g}.
\end{equation*}
Obviously $\xi_M$ is surjective. For s subring $F$ of $\BC$, we let ${\sf Sing}(Mp_2(\BZ),\rho_M)_F$ be the
$F$-submodule of ${\sf Sing}(Mp_2(\BZ),\rho_M)$ for which the coefficients of $q^n\,{\mathfrak e}_\g$ for
$n\leq 0$ are in $F$, and let
\begin{equation*}
{\sf ModForm}\big( Mp_2(\BZ),1-{\ell}/2,\rho_M\big)_\BZ\,\subset\,
{\sf ModForm}\big( Mp_2(\BZ),1-{\ell}/2,\rho_M \big)
\end{equation*}
be the $\BZ$-submodule whose image under $\mu_{1-{\ell}/2}$ lies in ${\sf Sing}(Mp_2(\BZ),\rho_M)_\BZ$.

\vskip 3mm\noindent
{\bf Theorem\ C.2}\,
(Borcherds \cite{Bo2}). {\it Suppose $M$ is an even lattice of signature $(2,\ell)$. Let $f$ be a
nearly holomorphic modular form on $\BH$ of weight $1-{\ell}/2$ and type $\rho_M$ whose coefficients
$c_{n,\g}$ are {\it integers} for $n\leq 0$. Then the Heegner divisor $\sum_{n,\g}c_{n,\g}\,y_{n,\g}$ is
principal, in other words,
\begin{equation*}
\xi_M \left( \mu_{1-{\ell}/2}\left( {\sf ModForm}\big( Mp_2(\BZ),1-{\ell}/2,\rho_M \big) \right) \right)
\subset {\sf PrinHeeg}(X_\G).
\end{equation*}}
\vskip 0.2cm \noindent
{\it Proof.}  The proof can be found in \cite[Theorem 4.1,\,\,pp.\,225--226]{Bo2}. \hfill $\square$

\vskip 3mm\noindent
{\bf Lemma\ C.1.}
{\it There is a number field $F$ of finite degree over $\BQ$ such that the space
$${\sf HolModForm}\big( Mp_2(\BZ),1+{\ell}/2,\rho_M^* \big)$$ has a basis $\{ f_1,\cdots,f_d\}$ whose Fourier
coefficients all lie in $F$, i.e., such that
$$\lambda_{1+{\ell}/2}^*(f_i)\in {\sf PowSer}\big( Mp_2(\BZ),\rho_M^*\big)_F$$
for all $i=1,2,\cdots,d.$}

\vskip 0.2cm
\vskip 3mm\noindent
{\bf Lemma\ C.2.}
{\it Let
$${\rm Gal}(\overline{\BQ}/\BQ) \cdot \lambda_{1+{\ell}/2}^*
\left(  {\sf HolModForm}\big( Mp_2(\BZ),1+{\ell}/2,\rho_M^* \big) \right)$$
be the space of ${\rm Gal}(\overline{\BQ}/ \BQ)$-conjugates of the $q$-expansions of elements of
$${\sf HolModForm}\big( Mp_2(\BZ),1+{\ell}/2,\rho_M^* \big).$$ Then
\begin{eqnarray*}
& & \mu_{1-{\ell}/2}\left( {\sf ModForm}\big( Mp_2(\BZ),1-{\ell}/2,\rho_M^* \big)_\BZ \right)\otimes\BC \subset
{\sf Sing}(Mp_2(\BZ),\rho_M)\\
&=& \left( {\rm Gal}(\overline{\BQ}/\BQ) \cdot \lambda_{1+{\ell}/2}^*
\left(  {\sf HolModForm}\big( Mp_2(\BZ),1+{\ell}/2,\rho_M^* \big) \right) \right)^\perp.
\end{eqnarray*}
Moreover this space has finite index in ${\sf Sing}(Mp_2(\BZ),\rho_M)$.}
\vskip 0.2cm\noindent
{\it Proof.} Using Theorem C.1 and Lemma C.1, we can prove Lemma C.2. We refer to
\cite[Lemma 4.3, pp.\,226--227]{Bo2} for
detail. \hfill $\square$

\vskip 0.2cm
Finally Borcherds\,\cite[Theorem 4.5]{Bo2} proves a generalization of the Gross-Kohnen-Zagier theorem to higher dimension.

\vskip 3mm\noindent
{\bf Theorem\ C.3.}
{\it The series
$${\rm (C.5)}\hskip 3cm
\sum_{n\in \BQ}\sum_{\g\in M'/M} y_{-n,\g}\,q^n\,{\mathfrak e}_\g \hskip 10cm
$$
lies in the space
\begin{equation*}
\left( {\sf HeegCl}(X_\G)\otimes\BC\right) \otimes_\BC
\left( {\rm Gal}(\overline{\BQ}/\BQ) \cdot \lambda_{1+{\ell}/2}^*
\left(  {\sf HolModForm}\big( Mp_2(\BZ),1+{\ell}/2,\rho_M^* \big) \right) \right),
\end{equation*}
that is, the series (C.5) is a modular form of weight $1+{\ell}/2$ and type $\rho_M$.}
\vskip 0.2cm\noindent
{\it Proof.} From the fact that the complex vector space ${\sf HeegCl}(X_\G)\otimes\BC$ generated by
the Heegner divisor classes is finite dimensional, we see that
\begin{equation*}
\sum_{n\in \BQ}\sum_{\g\in M'/M} y_{-n,\g}\,q^n\,{\mathfrak e}_\g \in
\left( {\sf HeegCl}(X_\G)\otimes\BC\right) \otimes_\BC
{\sf PowSer}\big( Mp_2(\BZ),\rho_M^*\big).
\end{equation*}
By Theorem C.2, the pairing
$$\sum_{n,\g}c_{n,\g}\,y_{n,\g}=\xi_M \left( \mu_{1-{\ell}/2}
\left( \sum_{n,\g}q^n\,{\mathfrak e}_\g\right) \right)$$
is zero. Hence the above theorem follows from Lemma C.2. \hfill $\square$


\end{document}